\documentclass{amsart}
\usepackage[T1]{fontenc}
\usepackage{mathrsfs}
\usepackage{amsfonts}
\usepackage{amssymb, amsthm, amsmath}
\usepackage{amsxtra}
\usepackage{dsfont}
\usepackage[colorlinks,citecolor=blue,urlcolor=blue]{hyperref}

\numberwithin{equation}{section}
\usepackage[english,polish]{babel}

\newcommand{\pl}[1]{\foreignlanguage{polish}{#1}}

\newtheorem{theorem}{Theorem}

\newtheorem{proposition}{Proposition}[section]

\newtheorem{lemma}{Lemma}

\newtheorem*{loglemma}{Bourgain's lemma}

\newcounter{thm}

\newtheorem{main_theorem}[thm]{Theorem}

\newcommand{\RR}{\mathbb{R}}
\newcommand{\ZZ}{\mathbb{Z}}
\newcommand{\TT}{\mathbb{T}}
\newcommand{\CC}{\mathbb{C}}
\newcommand{\NN}{\mathbb{N}}
\newcommand{\QQ}{\mathbb{Q}}

\newcommand{\calP}{\mathcal{P}}
\newcommand{\calQ}{\mathcal{Q}}
\newcommand{\calF}{\mathcal{F}}

\newcommand{\calI}{\mathcal{I}}
\newcommand{\calW}{\mathcal{W}}

\newcommand{\var}[1]{V_r{#1}}

\newcommand{\bvar}[1]{\mathcal{V}_r{#1}}

\newcommand{\seq}[2]{{#1}: {#2}}

\newcommand{\ind}[1]{{\mathds{1}_{{#1}}}}

\newcommand{\lcm}{\operatorname{lcm}}
\newcommand{\tr}{\operatorname{tr}}

\renewcommand{\atop}[2]{\substack{{#1}\\{#2}}}
\newcommand{\norm}[1]{{\left\lvert #1 \right\rvert}}
\newcommand{\sprod}[2] {{#1 \cdot #2}}

\newcommand{\abs}[1]{{\lvert {#1} \rvert}}
\newcommand{\sabs}[1]{{\left\lvert {#1} \right\rvert}}
\newcommand{\vnorm}[1]{{\left\lVert {#1} \right\rVert}}

\newcommand{\8}{\infty}

\newcommand{\vrho}{\varrho}

\title[Discrete maximal functions in higher dimensions]
{Discrete maximal functions in higher dimensions\\
 and applications to ergodic theory}

\author{Mariusz Mirek}
\address{Mariusz Mirek \\
	Universit\"{a}t Bonn \\
	Mathematical Institute\\
	Endenicher Allee 60\\
	D--53115 Bonn \\
	Germany \&
	Instytut Matematyczny\\
	Uniwersytet \pl{Wroc{\lll}awski}\\
	Pl. Grun\-waldzki 2/4\\
	50-384 \pl{Wroc{\lll}aw}\\
	Poland}
 \email{mirek@math.uni-bonn.de}

\author{Bartosz Trojan}
\address{
	Bartosz Trojan\\
	Instytut Matematyczny\\
	Uniwersytet \pl{Wroc{\lll}awski}\\
	Pl. Grun\-waldzki 2/4\\
	50-384 \pl{Wroc{\lll}aw}\\
	Poland}
\email{trojan@math.uni.wroc.pl}

\thanks{The authors were partially supported by NCN grant DEC--2012/05/D/ST1/00053}

\begin{document}
\selectlanguage{english}

\begin{abstract}
	We establish a higher dimensional counterpart of Bourgain's pointwise ergodic theorem
	along an arbitrary integer-valued polynomial mapping. We achieve this by proving
	variational estimates $V_r$ on $L^p$ spaces for all $1<p<\infty$ and $r>\max\{p, p/(p-1)\}$.
	Moreover, we obtain the estimates which are uniform in the coefficients of a polynomial mapping
	of fixed degree.
\end{abstract}

\maketitle

\section{Introduction}
In the mid 1980's Bourgain extended Birkhoff's pointwise ergodic theorem, proving  that for any
dynamical system $(X, \mathcal{B}, \mu, T)$ on a $\sigma$-finite measure space $X$ with an
invertible measure preserving transformation $T$ the averages along the squares
$$
A_N f(x)=N^{-1}\sum_{n=1}^Nf\big(T^{n^2}x\big)
$$
converge  $\mu$-almost everywhere on $X$ for all $f\in L^p(X, \mu)$ with $p>1$,
(see \cite{bou1, bou2}). Not long afterwards in \cite{bou}, the squares were replaced
by an arbitrary integer-valued polynomial. The restriction to the range $p>1$ in Bourgain's theorem
turned out to be essential. Recently, Buczolich and Mauldin \cite{BM} have shown that the pointwise
convergence of $A_Nf$ fails on $L^1(X,\mu)$ (see also \cite{LaV}).

In this article we are concerned with $L^p(X, \mu)$ estimates for discrete higher dimensional
analogues of the averaging operator and applications of such estimates to pointwise ergodic
theorems.

Let $(X, \mathcal{B}, \mu)$ be a $\sigma$-finite measure space with a family of invertible,
commuting and measure preserving transformations $T_1, T_2,\ldots,T_{d_0}$ for some $d_0\in\NN$.
Let $\calP=\big(\calP_1, \ldots, \calP_{d_0}\big): \ZZ^k \rightarrow \ZZ^{d_0}$ denote a polynomial
mapping such that each $\calP_j$ is an integer-valued polynomial on $\ZZ^k$ with $\calP_j(0) = 0$.
Define the averages
\begin{align}
	\label{eq:2}
	A_N^{\calP} f(x)
	=N^{-k} \sum_{n\in\NN_N^k}
	f\big(T_1^{\calP_1(n)}T_2^{\calP_2(n)}\cdot\ldots\cdot T_{d_0}^{\calP_{d_{0}}(n)}x\big)
\end{align}
where $\NN^k_N = \{1, 2, \ldots, N\}^k$. The results of this paper establish the following.
\begin{main_theorem}
	\label{thm:3}
	Assume that $p \in (1, \infty)$. For every $f \in L^p(X, \mu)$ there exists
	$f^*\in L^p(X, \mu)$ such that
	\begin{align*}
		\lim_{N\to\infty}A_{N}^{\calP}f(x)=f^*(x)
	\end{align*}
	$\mu$-almost everywhere on $X$.
\end{main_theorem}
Classical proofs of pointwise convergence require $L^p(X, \mu)$ bounds for maximal function,
reducing the problem to proving pointwise convergence for a dense class of $L^p(X, \mu)$ functions.
However, establishing pointwise convergence on a dense class may be a difficult problem.
For instance Bourgain's celebrated averaging operator along the squares is such an example.
One of the possibilities, introduced by Bourgain in \cite{bou}, for overcoming this problem
is to  control the $r$-variational seminorm $V_r$ of a sequence of measurable functions
$\big(\seq{f_j}{j \in \NN}\big)$ defined by
$$
\var{\big(\seq{f_j(x)}{j \in \NN}\big)}
=\sup_{\atop{k_0 < k_1 < \ldots < k_J}{k_j \in \NN}}
\bigg(\sum_{j = 1}^J \lvert f_{k_j}(x) - f_{k_{j-1}}(x) \rvert^r \bigg)^{1/r}.
$$
Indeed, if $\var{\big(\seq{f_j(x)}{j \in \NN}\big)} < \infty$ for some finite $r\ge1$
then the sequence $\big(\seq{f_j(x)}{j \in \NN}\big)$ converges. Theorem \ref{thm:3}, in particular, will follow from more general result. Namely.
 \begin{main_theorem}
	\label{thm:3'}
	Let $p \in (1, \8)$ and $r > \max\{p, p/(p-1)\}$. Then there is a constant $C_{p, r}>0$
	such that for every $f \in L^p(X, \mu)$
	\begin{align}
		\label{eq:2''}
		\big\lVert
		\var{\big(\seq{A_N^\calP f}{N \in \NN}\big)}
		\big\rVert_{L^p}
		\le C_{p, r}
		\vnorm{f}_{L^p}.
	\end{align}
	Moreover, the constant $C_{p, r}$ is independent of the coefficients of the polynomial mapping
	$\calP$.
\end{main_theorem}
In view of Calder\'{o}n's transference principle, one can reduce our problem and work on
$\ZZ^{d_0}$ rather than on an abstract measure space $X$. In this setting we  consider
the averages
\begin{align}
	\label{eq:1}
	M_N^\calP f(x) = N^{-k} \sum_{y \in \NN_N^k} f\big(x - \calP(y)\big)
\end{align}
for any finitely supported function $f: \ZZ^{d_0} \rightarrow \CC$. We will be mainly
interested in $\ell^p$ bounds for $r$-variations of the averages $M_N^\calP$. In this setup
Theorem \ref{thm:3'} can be reformulated in the following way.
\begin{main_theorem}
	\label{thm:2}
	Let $p \in (1, \8)$ and $r > \max\{p, p/(p-1)\}$. There is a constant $C_{p, r}>0$
	such that for every $f \in \ell^p\big(\ZZ^{d_0}\big)$
	\begin{align}
		\label{eq:2'}
		\big\lVert
		\var{\big(\seq{M_N^\calP f}{N \in \NN}\big)}
		\big\rVert_{\ell^p}
		\le C_{p, r}
		\vnorm{f}_{\ell^p}.
	\end{align}
	Moreover, the constant $C_{p, r}$ is independent of the coefficients of the polynomial mapping
	$\calP$.
\end{main_theorem}
Theorem \ref{thm:2} is the main result of this article and generalizes recent one dimensional
variational estimates of Krause \cite{K}. However, its proof will strongly explore maximal theorem
for $M_N^\calP$. Namely, Theorem \ref{thm:1} which is the higher dimensional counterpart of
Bourgain's theorem \cite{bou}.
\begin{main_theorem}
	\label{thm:1}
	For each $p \in (1, \8]$ there is a constant $C_p > 0$ such that for every
	$f \in \ell^p\big(\ZZ^{d_0}\big)$
\begin{align}\label{eq:1'}
  \big\lVert
	\sup_{N \in \NN}
	\big\lvert M_N^\calP f \big\rvert
	\big\rVert_{\ell^p}
	\leq
	C_p \vnorm{f}_{\ell^p}.
\end{align}
	Moreover, the constant $C_p$ is independent of the coefficients of the polynomial mapping
	$\calP$.
\end{main_theorem}
Bourgain's papers \cite{bou1, bou2, bou} initiated extensive study both in pointwise ergodic
theory along various arithmetic subsets of the integers
(see e.g. \cite{BKQW, kev, IMSW, K, M2, M1, wrl}) and investigations of discrete analogues of
classical operators with arithmetic features
(see e.g. \cite{IMSW, I,  IMW, IW, MSW, MSW1, mt2, O, P, P2,  P1, SW3, SW0, SW1, SW2}).
Variational inequalities in harmonic analysis and ergodic theory have been the subject of many
recent articles, see especially \cite{JKRW, JSW, K, OSTTW, zk} and the references given there
(see also \cite{NOT, Ob}).

We were motivated to study pointwise convergence of the averaging operators defined in \eqref{eq:2}
by recent results of Ionescu, Magyar, Stein and Wainger \cite{IMSW}. They
considered pointwise convergence of some noncommutative variants of averaging operators along
the polynomials of degree at most $2$. The desire to better understand the restriction
imposed on the degree of polynomials in \cite{IMSW}, has led to, in particular,
Theorem \ref{thm:2} and Theorem \ref{thm:1}  from this paper. Furthermore, the recent paper of
Krause \cite{K} inspired us to study variational estimates in higher dimensions
--- see Theorem \ref{thm:2} --- which in turn provide an approach to pointwise convergence
different to the argument from \cite{IMSW}. Specifically, in this paper we relax the restriction on
the degree of polynomials and we obtain all the results (maximal and variational estimates and
pointwise convergence) for polynomials of arbitrary degree at the expense of the loss of the
noncommutative setup which was the subject of \cite{IMSW}.

The purpose of this article, compared with the prior works, is threefold. Firstly, as we said
before, we shall relax the restriction for the degree of the polynomials from \cite{IMSW}.
Secondly, we provide variational estimates and thirdly, we will establish bounds in the
inequalities \eqref{eq:2'} and \eqref{eq:1'} which are uniform in the coefficients of underlying
polynomial mapping. The last statement finds applications in the discrete multi-parameter theories
of maximal functions and singular integral operators.

The inequality from \eqref{eq:1'} turned out to be decisive in one parameter theory, for instance  in the ongoing project concerning $\ell^p$ estimates for the maximal function corresponding with truncations of Radon transform from \cite{IW}. Namely,  we have recently established, for some  $p>1$, the following inequality
\begin{align}\label{bounds2}
  \big\|\sup_{N\in\NN}\big|T_N^{\mathcal{P}}f(x)\big|\big\|_{\ell^p}\le C_p\|f\|_{\ell^p}
 \end{align}
 where $T_N^{\mathcal{P}}f$ is a truncated Radon transform along the polynomial mapping $\mathcal{P}$, i.e.
$$
T_N^{\mathcal{P}}f(x)=\sum_{y\in\ZZ_N^k\setminus\{0\}}f(x-\mathcal{P}(y))K(y),
$$
where $K$ is a Calder\'{o}n-Zygmund kernel on $\RR^k$ and $\ZZ^k_N=\{-N, \ldots, -1, 0, 1, \ldots, N\}^k$.
In fact, in the proof of inequality \eqref{bounds2} we had to replace the supremum over the set of integers $\NN$ with the supremum over   the set of dyadic numbers $\{2^n: n\in\NN\cup\{0\}\}$. Since the operators $T_N^{\mathcal{P}}$ are not positive we had to be more careful, but for $N\in[2^n, 2^{n+1})$ we have the pointwise estimate
$$
\big|T_N^{\mathcal{P}}f(x)\big|\le C \big(\big|T_{2^n}^{\mathcal{P}}f(x)\big|+M_N^{\mathcal{P}}|f|(x)\big)
$$
for some $C>0$.

The proof of Theorem \ref{thm:1} will be based on an idea of Ionescu and Wainger from
\cite{IW} where they established $\ell^p$ bounds for the discrete Radon transform by partitioning
the operator into two parts, the first part controllable in $\ell^p$ and the second part
controllable in $\ell^2$. More precisely, for every $\epsilon \in (0, 1]$ and $\lambda>0$ we are
going to find an operator $A^{\lambda, \epsilon}_N$ such that
\begin{align*}
  \big\lVert
	\sup_{N \in \NN}
	\big\lvert M_N^\calP f - A_N^{\lambda, \epsilon} f \big\rvert
	\big\rVert_{\ell^2}
	\leq
	D_\epsilon \lambda^{-1} \vnorm{f}_{\ell^2}
\end{align*}
and for each $p \in (1, \infty)$
\begin{align*}
  \big\lVert
	\sup_{N \in \NN}
	\big\lvert A_N^{\lambda, \epsilon} f \big\rvert
	\big\rVert_{\ell^p}
	\leq
	C_\epsilon \lambda^{\epsilon} \vnorm{f}_{\ell^p}.
\end{align*}
Then with the aid of these two estimates one can use restricted
interpolation techniques as in \cite{IW} and conclude that \eqref{eq:1'} holds. The same idea was
also explored in \cite{IMSW}. Here we are going to make use of this argument and provide a
different approach to the estimates in $\ell^2$ and $\ell^p$ as compared both to Bourgain's paper
\cite{bou} and Ionescu, Magyar, Stein and Wainger's paper \cite{IMSW}. Since the $r$-variational
seminorm controls the supremum norm for any $r\ge1$ we only need $r$-variational estimates on
$\ell^2$. The $\ell^2$ theory for averaging operators along polynomials in \cite{bou} was built,
to a large extent, on the circle method of Hardy and Littlewood and on the ``logarithmic'' lemma
due to Bourgain (see \cite{bou}, see also \cite{lac1}).
\begin{loglemma}
	Assume that $\lambda_1<\ldots<\lambda_K \in \RR$ and for $j\in\NN$ define the neighbourhoods
	$$
	\mathcal{R}_j=\{\xi\in\RR: \min_{1\le k\le K}|\xi-\lambda_k|\le 2^{-j}\}.
	$$
	Then there exists a  constant $C>0$ such that
	\begin{align*}
		\Big\|\sup_{j\in\NN}
		\Big|\int_{\mathcal{R}_j}\hat{f}(\xi)e^{2\pi i \xi x}d\xi\Big|\Big\|_{L^2(dx)}
		\le C(\log K)^3\|f\|_{L^2},
	\end{align*}
	for every $f \in L^2(\RR)$.
\end{loglemma}
Although Bourgain's lemma is interesting in its own right, and is a powerful tool in discrete
problems, it has also found  wide application in problems susceptible to time-frequency analysis (see e.g.
\cite{DLTT, Lac, Thi}). Recently, Nazarov, Oberlin and Thiele \cite{NOT} introduced a
multi-frequency Calder\'{o}n-Zygmund decomposition and extended Bourgain's estimates
to provide $L^p$ bounds and variational estimates (see also \cite{Ob}). Some refinement of the
results from \cite{NOT} established by Krause \cite{K2} turned out  to be
an invaluable tool in variational estimates for Bourgain's averages along polynomials in \cite{K}. In this
paper we propose another approach which avoids the use of Bourgain's lemma
or its counterparts.

Our approach to Theorem \ref{thm:1} and Theorem \ref{thm:2} proceeds in several stages. We begin
with a particular lifting of the operator \eqref{eq:1}. This procedure will permit us to replace
any polynomial mapping $\calP$ by a new polynomial mapping (the canonical polynomial mapping:
see Lemma \ref{lem:1} in Section \ref{sec2}) which has all coefficients equal to $1$. This will
result in the uniform estimates and will reduce the study to the canonical polynomial mapping.
In Section \ref{sec3} we construct suitable approximating multipliers: see the definitions of
\eqref{eq:22}, \eqref{eq:71} and \eqref{eq:72}, and prove  strong $\ell^2$ bounds on their
$r$-variations. These multipliers will be useful in proving Theorem \ref{thm:1} and
Theorem \ref{thm:2} for $p=2$ in Section \ref{sec4} and Section \ref{sec5}, respectively.
The proofs of these $\ell^2$ bounds, on the one hand, will be covered by the multi-dimensional
variant of the circle method of Hardy and Littlewood.

On the other hand, we will make use of some elementary inequalities for $r$-variations from  Section \ref{sec2}, which have not been used in this context. This is the novelty of the paper
and allows us to study the approximating multipliers by a direct analysis which avoids using
results like Bourgain's lemma.

In Section \ref{sec3} we also provide the $\ell^p$ theory,
$p>1$, necessary to obtain Theorem \ref{thm:1}. The strategy of the proof of $\ell^p$ bounds will
be very simple. We shall compare the discrete norm $\| \cdot\|_{\ell^p}$ of our approximating
multipliers with  the continuous norm $\| \cdot\|_{L^p}$ of certain multipliers which are a priori
bounded on $L^p$. But this will only give good bounds when $N$ is restricted to the large cubes
depending on $\lambda$ as in the Ionescu--Wainger partition. The small cubes will be covered by
a restricted $\ell^p$ bound with logarithmic loss for the operator \eqref{eq:1}: see
Theorem \ref{th:5}. This idea was pioneered by Bourgain in \cite{bou} to prove the full range of
$\ell^p$ estimates. Here we will explore this idea, giving a slightly simpler proof of this fact.
All these results will allow us to decompose the operator $M_N^{\calP}$ into two parts
$A_{N, \varepsilon}^{\lambda}$ and $M_N^{\calP}-A_{N, \varepsilon}^{\lambda}$ as was described
above and will establish Theorem \ref{thm:1}: see Section \ref{sec4}. Finally, having proved
Theorem \ref{thm:1} for all $1<p\le \infty$ and Theorem \ref{thm:2} for $p=2$ and $2<r<\infty$ we
shall employ the interpolation argument from Krause's paper \cite{K} and  conclude that
Theorem \ref{thm:2} holds for all $1<p<\infty$ and $r>\max\{p, p/(p-1)\}$.

\subsection{Notation}
Throughout the whole article, unless otherwise stated, we will write $A \lesssim B$
($A \gtrsim B$) if there is an absolute constant $C>0$ such that $A\le CB$ ($A\ge CB$).
Moreover, $C > 0$ will stand for a large positive constant whose value may vary from occurrence to
occurrence. If $A \lesssim B$ and $A\gtrsim B$ hold simultaneously then we will write
$A \simeq B$. Lastly, we will write $A \lesssim_{\delta} B$ ($A \gtrsim_{\delta} B$) to
indicate that the constant $C>0$ depends on some $\delta > 0$. Let $\NN_0 = \NN \cup \{0\}$.
For a vector $x \in \RR^d$ we set $\norm{x} = \max\{\abs{x_j} : 1 \leq j \leq d\}$ and
$\mathcal{D}=\{2^n: n\in\NN_0\}$ will denote the set of dyadic numbers.

\section{Preliminaries}\label{sec2}

\subsection{Variational norm}
Let $1 \leq r < \infty$. For each sequence $\big(\seq{a_j}{j \in A}\big)$ where $A\subseteq\ZZ$
we define $r$-variational seminorm by
$$
\var{\big(\seq{a_j}{j \in A}\big)} = \sup_{\atop{k_0 < k_1 < \ldots < k_J}{k_j \in A}}
	\bigg(\sum_{j = 1}^J \abs{a_{k_j} - a_{k_{j-1}}}^r \bigg)^{1/r}.
$$
The function $r \mapsto \var{\big(\seq{a_j}{j \in A}\big)}$ is non-increasing and satisfies
\begin{equation}
	\label{eq:25}
	\sup_{j \in A} \abs{a_j} \leq
	\var{\big(\seq{a_j}{j \in A}\big)} + \abs{a_{j_0}}
\end{equation}
where $j_0$ is an arbitrary element of $A$. Let
$$
\bvar{\big(\seq{a_j}{j \in A}\big)}=\sup_{j\in A}|a_j|+\var{\big(\seq{a_j}{j \in A}\big)}.
$$
For any subset $B \subseteq A$ we have
$$
\var{\big(\seq{a_j}{j \in B}\big)} \leq \var{\big(\seq{a_j}{j \in A}\big)}.
$$
Moreover, if $-\infty\le u < w < v \leq \infty$ then
\begin{equation}
	\label{eq:19}
	\var{\big(\seq{a_j}{u < j < v}\big)}
	\leq
	2 \sup_{u < j < v} \abs{a_j}
	+ \var{\big(\seq{a_j}{u < j < w}\big)}
	+ \var{\big(\seq{a_j}{w < j < v}\big)}.
\end{equation}
For $r \geq 2$ we also have
\begin{equation}
	\label{eq:18}
	\var{\big(\seq{a_j}{j \in A}\big)}
	\leq 2 \Big(\sum_{j \in A} \abs{a_j}^2\Big)^{1/2}.
\end{equation}
We will need the following simple observation.
\begin{lemma}
	\label{lem:6}
	If $r \geq 2$ then for any sequence $\big(\seq{a_j}{0 \leq j \leq 2^s}\big)$
	\begin{equation}
		\label{eq:62}
		\var{\big(\seq{a_j}{0 \leq j \leq 2^s}\big)}
		\leq
		\sqrt{2}
		\sum_{i = 0}^s
		\bigg(
		\sum_{j = 0}^{2^{s-i}-1}
		\sabs{a_{(j+1)2^i} - a_{j 2^i}}^2
		\bigg)^{1/2}.
	\end{equation}
\end{lemma}
\begin{proof}
	Let us observe that any interval $[m, n)$ for $m, n \in \NN$ such that $0 \leq m < n \leq 2^s$,
	is a finite disjoint union of dyadic subintervals, i.e. intervals belonging to some $\calI_i$ for
	$0 \leq i \leq s$, where
	$$
	\calI_i = \big\{[j2^i, (j+1)2^i): 0 \leq j \leq 2^{s-i}-1\}
	$$
	and such that each length appears at most twice. For the proof, we set $m_0 = m$. Having
	chosen $m_p$ we select $m_{p+1}$ in such a way that $[m_p, m_{p+1})$ is the longest dyadic
	interval starting at $m_p$ and contained inside $[m_p, n)$. If the lengths of the selected
	dyadic intervals increase then we are done. Otherwise, there is $p$ such that
	$m_{p+1} - m_p \geq m_{p+2} - m_{p+1}$. We show that this implies
	$m_{p+2} - m_{p+1} > m_{p+3} - m_{p+2}$. Suppose for a contradiction that,
	$m_{p+2} - m_{p+1} \leq m_{p+3} - m_{p+2}$. Then
	$$
	[m_{p+1}, 2 m_{p+2} - m_{p+1}) \subseteq [m_{p+1}, m_{p+3}).
	$$
	Therefore, it is enough to show that $2(m_{p+2} - m_{p+1})$ divides $m_{p+1}$. It is clear
	in case $m_{p+1} - m_p > m_{p+2} - m_{p+1}$. If $m_{p+1} - m_p = m_{p+2} - m_{p+1}$ then,
	by maximality of $[m_p, m_{p+1})$, $2(m_{p+2} - m_{p+1})$ cannot divide $m_p$, thus divides
	$m_{p+1}$.
	
	Next, let $k_0 < k_1 < \ldots < k_J \leq 2^s$ be any increasing sequence. For
	each $j \in \{0, \ldots, J-1\}$ we may write
	$$
	[k_j, k_{j+1}) = \bigcup_{p=0}^{P_j} [u_p^j, u_{p+1}^j)
	$$
	for some $P_j\ge1$ where each interval $[u_p^j, u^j_{p+1})$ is dyadic. Then
	$$
	\lvert
	a_{k_{j+1}} - a_{k_j}
	\rvert
	\leq
	\sum_{p = 0}^{P_j}
	\big \lvert
	a_{u_{p+1}^j} - a_{u_p^j}
	\big \rvert
	=
	\sum_{i = 0}^s
	\sum_{p:\: [u_p^j, u_{p+1}^j) \in \calI_i}
	\big\lvert
	a_{u_{p+1}^j} - a_{u_p^j}
	\big\rvert.
	$$
	Hence, by Minkowski's inequality
	\begin{multline*}
		\Big(
		\sum_{j = 0}^{J-1}
		\sabs{a_{k_{j+1}} - a_{k_j}}^2
		\Big)^{1/2}
		\leq
		\Big(
		\sum_{j = 0}^{J-1}
		\Big(
		\sum_{i = 0}^s
		\sum_{p:\: [u_p^j, u_{p+1}^j) \in \calI_i}
		\big\lvert a_{u^j_p} - a_{u^j_{p+1}} \big\rvert
		\Big)^2
		\Big)^{1/2}\\
		\leq
		\sum_{i = 0}^s
		\Big(
		\sum_{j = 0}^{J-1}
		\Big(
		\sum_{p:\: [u_p^j, u_{p+1}^j) \in \calI_i}
		\big\lvert a_{u^j_p} - a_{u^j_{p+1}} \big\rvert
		\Big)^2
		\Big)^{1/2}.
	\end{multline*}
	Since for a given $i \in \{0, 1, \ldots, 2^s\}$ and $j \in \{0, 1, \ldots, J-1\}$
	the inner sums contain at most two elements we obtain
	$$
	\Big(
	\sum_{j = 0}^{J-1}
	\sabs{a_{k_{j+1}} - a_{k_j}}^2
	\Big)^{1/2}
	\le
	\sqrt{2}
	\sum_{i = 0}^s
	\Big(
	\sum_{j = 0}^{J-1}
	\sum_{p:\: [u_p^j, u_{p+1}^j) \in \calI_i}
	\big\lvert a_{u^j_p} - a_{u^j_{p+1}} \big\rvert^2
	\Big)^{1/2}
	$$
	what is bounded by the right-hand side of \eqref{eq:62}.
\end{proof}
A \emph{long variation seminorm} $V_r^L$ of a sequence $\big(\seq{a_j}{j \in A}\big)$, is given by
\begin{align*}
	V_r^L\big(a_j: j\in A\big)=V_{r}\big(a_j: j\in A\cap\mathcal{D}\big).
\end{align*}
A \emph{short variation seminorm} $V^S_r$ is given by
\begin{align*}
	V^S_r\big(a_j: j\in A\big)
	=
	\bigg(\sum_{n \ge 0} V_r\big(a_j: j \in A_n\big)^r\bigg)^{1/r}
\end{align*}
where $A_n = A \cap [2^n, 2^{n+1})$. Then
\begin{align}
	\label{eq:32}
	V_r\big(a_j: j\in \NN\big)
	\lesssim
	V_r^L\big(a_j: j\in \NN\big)
	+V_r^S\big(a_j: j\in \NN\big).
\end{align}
The next lemma will be used in the estimates for short variations. It illustrates the ideas
which have been explored several times (see \cite{JKRW}, or recently \cite{K}).
\begin{lemma}
	\label{lem:5}
	Let $u, v \in \NN$, $u < v$. For any integer $h \in \{1,\ldots, v-u\}$ there is a strictly
	increasing sequence of integers $\big(\seq{m_j}{0 \leq j \leq h}\big)$ with $m_0 = u$ and
	$m_h = v$ such that for every $r\ge2$
	$$
	V_r\big(a_j: u \leq j \leq v\big)
	\lesssim
	\bigg(\sum_{j=0}^h|a_{m_j}|^2\bigg)^{1/2}
	+\big((v-u)/h\big)^{1/2}\bigg(\sum_{j=u}^{v-1}|a_{j+1}-a_j|^2\bigg)^{1/2}.
	$$
\end{lemma}
\begin{proof}
	It is enough to consider $r = 2$. Fix $h \in \{1,\ldots, v-u\}$ and choose a sequence
	$\big(\seq{m_j}{1 \leq j \leq h}\big)$ such that
	$m_0 = u$, $m_h = v$ and $|m_{j+1}-m_j| \simeq (v-u)/h$. Then
	\begin{multline*}
		V_2\big(a_j: u \leq j \leq v\big)
		\lesssim \Big(\sum_{j=0}^h|a_{m_j}|^2\Big)^{1/2}
		+\Big(\sum_{j=0}^{h-1}V_2\big(a_k: m_j \leq k \le m_{j+1}\big)^2\Big)^{1/2}\\
		\lesssim
		 \Big(\sum_{j=0}^h|a_{m_j}|^2\Big)^{1/2}
		+\Big(\sum_{j=0}^{h-1}\Big(\sum_{k=m_j}^{m_{j+1}-1}|a_{k+1}-a_k|\Big)^2\Big)^{1/2}.
	\end{multline*}
	By the Cauchy--Schwarz inequality the last sum can be dominated by
	\begin{multline*}
		\Big(\sum_{j=0}^h|a_{m_j}|^2\Big)^{1/2}
		+\Big(\sum_{j=0}^{h-1}(m_{j+1}-m_j)\sum_{k=m_j}^{m_{j+1}-1}|a_{k+1}-a_k|^2\Big)^{1/2}\\
		\lesssim
		\Big(\sum_{j=0}^h|a_{m_j}|^2\Big)^{1/2}
		+\big((v-u)/h\big)^{1/2}\Big(\sum_{j=u}^{v-1}|a_{j+1}-a_j|^2\Big)^{1/2}
\end{multline*}
and this completes the proof of the lemma.
\end{proof}
We observe that, if $(f_j: j\in\NN)$ is a sequence of functions in $\ell^2$ and $v-u\ge2$ then
\begin{align}
	\label{eq:39}
	\big\|V_r\big(f_j: j\in [u, v]\big)\big\|_{\ell^2}
	\lesssim
	\max\big\{A,
	(v-u)^{1/2} A^{1/2} B^{1/2}\big\}
\end{align}
where
$$
A =	\max_{u \leq j \leq v}\|f_j\|_{\ell^2},
\quad \text{and} \quad
B = \max_{u \leq j < v}\|f_{j+1}-f_j\|_{\ell^2}.
$$
Indeed, let
$$
h=\big\lceil (v-u) B / (4A)\big\rceil.
$$
Then $h \in [1, v- u]$. If $h\ge2$ we may estimate
$$
\big\|V_r\big(f_j: u \leq j \leq v\big)\big\|_{\ell^2}
\lesssim
A h^{1/2} + (v-u) B /h^{1/2}
\lesssim
(v-u)^{1/2} A^{1/2} B^{1/2}
$$
where the last inequality follows from Lemma \ref{lem:5}. If $h=1$ then
$B \lesssim (v - u)^{-1} A$ and hence
$$
\big\|V_r\big(f_j: u \leq j \leq v\big)\big\|_{\ell^2}
\lesssim
A.
$$

\subsection{Lifting lemma}
Let $\calP=(\calP_1,\ldots, \calP_{d_0}): \ZZ^k \rightarrow \ZZ^{d_0}$ be a mapping whose
components $\calP_j$ are integer valued polynomials on $\ZZ^k$ such that $\calP_j(0) = 0$. We set
$$
N_0 = \max\{ \deg \calP_j : 1 \leq j \leq d_0\}.
$$
It is convenient to work with the set
$$
\Gamma =
\big\{
	\gamma \in \ZZ^k \setminus\{0\} : 0 \leq \gamma_j \leq N_0
	\text{ for each } j = 1, \ldots, k
\big\}
$$
with the lexicographic order. Then each $\calP_j$ can be expressed as
$$
\calP_j(x) = \sum_{\gamma \in \Gamma} c_j^\gamma x^\gamma
$$
for some $c_j^\gamma \in \ZZ$. Let us denote by $d$ the cardinality of the set $\Gamma$.
We identify $\RR^d$ with the space of all vectors whose coordinates are labelled by multi-indices
$\gamma \in \Gamma$. Let $A$ be a diagonal $d \times d$ matrix such that
$$
(A v)_\gamma = \abs{\gamma} v_\gamma.
$$
For $t > 0$ we set
$$
t^{A}=\exp(A\log t)
$$
i.e. $t^A x=(t^{|\gamma|}x_{\gamma}: \gamma\in \Gamma)$. Next, we introduce the \emph{canonical}
polynomial mapping
$$
\calQ = \big(\seq{Q_\gamma}{\gamma \in \Gamma}\big) : \ZZ^k \rightarrow \ZZ^d
$$
where $\calQ_\gamma(x) = x^\gamma$ and $x^\gamma=x_1^{\gamma_1}\cdot\ldots\cdot x_k^{\gamma_k}$.
The coefficients $\big(\seq{c_j^\gamma}{\gamma \in \Gamma, j \in \{1, \ldots, d_0\}}\big)$ define
a linear transformation $L: \RR^d \rightarrow \RR^{d_0}$ such that $L\calQ = \calP$. Indeed, it is
enough to set
$$
(L v)_j = \sum_{\gamma \in \Gamma} c_j^\gamma v_\gamma
$$
for each $j \in \{1, \ldots, d_0\}$ and $v \in \RR^d$. The next lemma, inspired by the continuous analogue (see \cite{deL} or \cite[p. 515]{bigs})
reduces proofs of Theorem \ref{thm:1} and Theorem \ref{thm:2} to the canonical polynomial mapping.
\begin{lemma}
	\label{lem:1}
	Suppose that for some $p \in (1, \infty)$ and $r > 2$
	$$
	\big\lVert
	\bvar{\big(\seq{M_N^\calQ f}{N \in \NN}\big)}
	\big\rVert_{\ell^p(\ZZ^d)}
	\leq C_{p, r}
	\vnorm{f}_{\ell^p(\ZZ^d)}.
	$$
	Then
	\begin{equation}
		\label{eq:6}
		\big\lVert
		\bvar{\big(\seq{M_N^\calP f}{N \in \NN}\big)}
		\big\rVert_{\ell^p(\ZZ^{d_0})}
		\leq
		C_{p, r}
		\vnorm{f}_{\ell^p(\ZZ^{d_0})}.
	\end{equation}
\end{lemma}
\begin{proof}
	Let $R > 0$ and $\Lambda > 0$ be fixed. Let $f \in \ell^p\big(\ZZ^{d_0}\big)$. In the proof we
	let $x \in \ZZ^{d_0}$, $y \in \ZZ^k$ and $u \in \ZZ^d$. For any $x \in \ZZ^{d_0}$ we define a
	function $F_x$ on $\ZZ^d$ by
	$$
	F_x(z) =
	\begin{cases}
		f(x + L(z)) & \text{ if } \norm{z} \leq R + \Lambda^{k N_0},\\
		0 & \text{ otherwise.}
	\end{cases}
	$$
	If $\norm{y} \leq N$ and $\norm{u} \leq R$ then $\norm{u - \calQ(y)} \leq R + \Lambda^{kN_0}$.
	Therefore for each $x \in \ZZ^{d_0}$
	$$
	M_N^\calP f(x+L u)
	=
	\frac{1}{N^k} \sum_{y \in \NN_N^k} f\big(x + L\big(u - \calQ(y)\big)\big)
	=
	M_N^\calQ F_x(u).
	$$
	Hence,
	\begin{multline*}
		\big\lVert
		\bvar{\big(\seq{M^\calP_N f}{N \in [1, \Lambda]}\big)}
		\big\rVert_{\ell^p(\ZZ^{d_0})}^p
		=
		\frac{1}{(2R+1)^d}
		\sum_{x \in \ZZ^{d_0}}
		\sum_{\norm{u} \leq R}
		\Big(
		\bvar{\big(\seq{M^\calP_N f(x+Lu) }{N \in [1, \Lambda]}\big)}
		\Big)^p\\
		=
		\frac{1}{(2R+1)^d}
		\sum_{x \in \ZZ^{d_0}}
		\sum_{\norm{u} \leq R}
		\Big(
		\bvar{\big(\seq{M^\calQ_N F_x(u) }{N \in [1, \Lambda]}\big)}
		\Big)^p
		\leq
		C_{p, r}^p
		\frac{1}{R^d}
		\sum_{x \in \ZZ^{d_0}}
		\sum_{u \in \ZZ^d}
		\lvert F_x(u) \rvert^p
	\end{multline*}
	where in the last inequality we have used
	$$
	\big\lVert
	\bvar{\big(\seq{M_N^\calQ g}{N \in [1, \Lambda]}\big)}
	\big\rVert_{\ell^p(\ZZ^d)}
	\leq
	C_{p, r}
	\vnorm{g}_{\ell^p(\ZZ^d)}
	$$
	for any $g \in \ell^p\big(\ZZ^d\big)$. Since
	$$
	\sum_{x \in \ZZ^{d_0}} \sum_{u \in \ZZ^d} \abs{F_x(u)}^p
	=
	\sum_{x \in \ZZ^{d_0}} \sum_{\norm{u} \leq R + \Lambda^{kN_0}}
	\abs{f(x + L u)}^p \leq \big(R + \Lambda^{kN_0}\big)^d \vnorm{f}_{\ell^p(\ZZ^{d_0})}^p
	$$
	we get
	$$
	\big\lVert
	\bvar{\big(\seq{M^\calP_N f}{N \in [1, \Lambda]}\big)}
	\big\rVert_{\ell^p(\ZZ^{d_0})}^p
	\leq
	C_{p, r}^p
	\bigg(1 + \frac{\Lambda^{kN_0}}{R}\bigg)^d
	\vnorm{f}_{\ell^p(\ZZ^{d_0})}^p.
	$$
	Taking $R$ approaching infinity we conclude
	$$
	\big\lVert
	\bvar{\big(\seq{M^\calP_N f}{N \in [1, \Lambda]}\big)}
	\big\rVert_{\ell^p(\ZZ^{d_0})}^p
	\leq
	C_{p, r}^p
	\vnorm{f}_{\ell^p(\ZZ^{d_0})}^p
	$$
	which by monotone convergence theorem implies \eqref{eq:6}.
\end{proof}
In the rest of the article by $M_N$ we denote the average for canonical polynomial mapping $\calQ$,
i.e. $M_N = M_N^\calQ$.

\subsection{Gaussian sums}
Given $q \in \NN$ we set $\NN_q = \{a \in \NN : 1 \leq a \leq q\}$. Let $A_q$ be the subset of
$a \in \NN_q^d$ such that
$$
\gcd\big(q, \gcd(a_\gamma : \gamma \in \Gamma)\big) = 1.
$$
For any $q \in \NN$ and $a \in \ZZ^d$ we define
$$
G(a/q) = q^{-k} \sum_{y \in \NN_q^k} e^{2\pi i \sprod{(a/q)}{\calQ(y)}}.
$$
Then there is $\delta > 0$ such that for any $a \in A_q$ (see \cite{SW0, IW})
\begin{equation}
	\label{eq:8}
	\abs{G(a/q)}
	\lesssim
	q^{-\delta}.
\end{equation}

\subsection{Fourier multipliers}
For a function $f \in L^1\big(\RR^d\big)$ let
$$
\calF{f}(\xi) = \int_{\RR^d} e^{2\pi i \sprod{\xi}{x}} f(x) dx
$$
be the Fourier transform of $f$. If $f \in \ell^1\big(\ZZ^d\big)$ let
$$
\hat{f}(\xi) = \sum_{x \in \ZZ^d} e^{2\pi i \sprod{\xi}{x}} f(x)
$$
be the discrete Fourier transform of $f$.
For any function $f: \ZZ^d \rightarrow \CC$ with a finite support we have
$$
M_N f(x) = K_N * f(x)
$$
where $K_N$ is a kernel defined by
\begin{align}
	\label{eq:34}
	K_N(x) = N^{-k} \sum_{y \in \NN_N^k} \delta_{\calQ(y)}
\end{align}
and $\delta_y$ denotes Dirac's delta at $y \in \ZZ^k$. Let $m_N$ denote the discrete  Fourier transform of $K_N$, i.e.
$$
m_N(\xi) = N^{-k} \sum_{y \in \NN_N^k} e^{2\pi i \sprod{\xi}{\calQ(y)}}.
$$
Finally, we define
$$
\Phi_N(\xi) = \int_{[0, 1]^k} e^{2\pi i \sprod{\xi}{\calQ(N y)}} dy.
$$
Using a multi-dimensional version of van der Corput lemma (see \cite{bigs, ccw}) we may
estimate
\begin{equation}
	\label{eq:9}
	\abs{\Phi_N(\xi)}
	\lesssim
	\min\big\{1, \norm{N^A \xi}^{-1/d} \big\}.
\end{equation}
Additionally, we have
\begin{equation}
	\label{eq:10}
	\abs{\Phi_N(\xi) - 1}
	\lesssim
	\min\big\{1, \norm{N^A \xi}\big\}.
\end{equation}

\section{Approximating multipliers}
\label{sec3}
The purpose of this section is to introduce  multipliers \eqref{eq:22}, \eqref{eq:71} and
\eqref{eq:72}. In the first two subsections we collect some $\ell^2(\ZZ^d)$ and $\ell^p(\ZZ^d)$
estimates. Then we apply these results to obtain unrestricted and restricted type inequalities
for our multipliers. The last two subsections provide bounds necessary to establish
Theorem \ref{thm:1} and Theorem \ref{thm:2}. Throughout the rest of the article the maximal
functions  will be initially defined for any nonnegative finitely supported function $f$
and unless otherwise stated $f$ is always such a function.

\subsection{$\ell^2$-theory}
We begin with some basic approximations of the multiplier $m_N$ forced by some multi-dimensional
variant of the circle method of Hardy and Littlewood.

We fix $N \geq 1$. For any $\alpha, \beta > 0$ we define a family of \emph{major arcs} by
$$
\mathfrak{M}_N
= \bigcup_{1 \leq q \leq N^\alpha} \bigcup_{a \in A_q} \mathfrak{M}_N(a/q)
$$
where
$$
\mathfrak{M}_N(a/q)
=
\big\{\xi \in \TT^d : \abs{\xi_\gamma - a_\gamma /q} \leq N^{- \abs{\gamma} + \beta}
\text{ for all } \gamma \in \Gamma
\big\}.
$$
The set $\mathfrak{m}_N=\TT^d\setminus \mathfrak{M}_N$ will be called \emph{minor arc}. We treat
the interval $[0, 1]^d$ as $d$-dimensional torus $\TT^d=\RR^d/\ZZ^d$.
\begin{proposition}
	\label{prop:1}
	For any $\kappa > 0$ there exists $C > 0$ such that if for some $1 \leq q \leq N^\alpha$
	and $a \in A_q$
	$$
	\Big\lvert
	\xi_\gamma - \frac{a_\gamma}{q}
	\Big\rvert
	\leq
	\kappa
	\cdot
	N^{-\abs{\gamma} + \beta}
	$$
	for all $\gamma \in \Gamma$ then
	$$
	\big\lvert
	m_N(\xi) - G(a/q) \Phi_N(\xi - a/q)
	\big\rvert
	\le C
	N^{-1/4},
	$$
	provided that $4(\alpha+\beta)<1$.
\end{proposition}
\begin{proof}
	Let $\theta = \xi - a/q$. If $y,r \in \NN^k$ are such that $y \equiv r \pmod q$
	then for each $\gamma \in \Gamma$
	$$
	\xi_\gamma y^{\gamma} \equiv \theta_\gamma y^\gamma + (a_\gamma/q) r^\gamma \pmod 1.
	$$
	Hence,
	$$
	N^{-k}\sum_{y \in \NN_N^k} e^{2 \pi i \sprod{\xi}{\calQ(y)}}
	=
	N^{-k}\sum_{r \in \NN_q^k} e^{2\pi i \sprod{(a/q)}{\calQ(r)}}
	\sum_{\atop{y \in \NN_N^k}{q \mid (y - r)}}
	e^{2\pi i \sprod{\theta}{\calQ(y)}}= G(a/q) \Phi_N(\xi - a/q)+O\big(N^{-1/4}\big).
	$$
	The last equality has been achieved by the mean value theorem, since $1 \leq q \leq N^\alpha$
	and $\abs{\theta_\gamma} \leq N^{- \abs{\gamma} + \beta}$ for every $\gamma\in \Gamma$.
\end{proof}

For any $s\in\NN$ we set
$$
\mathscr{R}_s = \big\{ a/q \in \QQ^d: 2^s \leq q < 2^{s+1} \mbox{ and } a \in A_q\big\}
$$
and $\mathscr{R}_0=\{0\}$. Let $\nu_N = \sum_{s \geq 0} \nu_N^s$ with a sequence of
multipliers $(\seq{\nu_N^s}{s \geq 0})$ given by
\begin{equation}
	\label{eq:22}
	\nu_N^s(\xi) = \sum_{a/q \in \mathscr{R}_s} G(a/q) \Phi_N(\xi - a/q) \eta_s(\xi - a/q)
\end{equation}
where $\eta_s(\xi) = \eta\big(10^{(s+1)A} \xi\big)$ and $\eta: \RR^d \rightarrow \RR$ is a smooth
function such that $0 \leq \eta(x) \leq 1$ and
\begin{align}\label{eq:56}
  \eta(x) =
\begin{cases}
	1 & \text{ for } \norm{x} \leq 1/4, \\
	0 & \text{ for } \norm{x} \geq 1/2.
\end{cases}
\end{align}
We may assume that $\eta$ is a convolution of two smooth nonnegative functions with supports
contained in $[-1/2, 1/2]^d$.
\begin{lemma}
	\label{lem:4}
	If $4(\alpha + \beta) < 1$ then there are $C > 0$ and $\delta_1 > 0$ such that for all
	$N \in \NN$
	$$
	\sup_{\xi \in \TT^d}
	\big\lvert
	m_N(\xi) - \nu_N(\xi)
	\big\rvert
	\leq C
	N^{-\delta_1}.
	$$
\end{lemma}
\begin{proof}
	Let us notice that for a fixed $s \in \NN \cup \{0 \}$ and $\xi \in \TT^d$ the sum
	\eqref{eq:22} consists of a single term. Indeed, otherwise there would be different
	$a/q, a'/q' \in \mathscr{R}_s$ such that $\eta_s(\xi - a/q) \neq 0$ and
	$\eta_s(\xi -a'/q') \neq 0$. Thus, for some $\gamma \in \Gamma$
	$$
	2^{-2 s - 2} \leq \frac{1}{qq'}
	\leq
	\Big\lvert
	\xi_\gamma - \frac{a_\gamma}{q}
	\Big\rvert
	+
	\Big\lvert
	\xi_\gamma - \frac{a_\gamma'}{q'}
	\Big\rvert
	\leq
	10^{-s-1}.
	$$
	\noindent
	{\textbf{Major arcs estimates:}} Suppose $\xi \in \mathfrak{M}_N(a/q)$ with
	$1 \leq q \leq N^\alpha$ and $a \in A_q$. Let $s_0$ be such that
	$$
	2^{s_0} \leq q < 2^{s_0+1}.
	$$
	We choose $s_1 \in \NN$ to satisfy
	$$
	2^{s_1 + 1} \leq N^{1 - \alpha - \beta} < 2^{s_1+2}
	$$
	If $s < s_1$ then for any $a'/q' \in \mathscr{R}_s$, $a'/q' \neq a/q$ and $\gamma \in \Gamma$
	we have
	$$
	\Big\lvert
	\xi_\gamma - \frac{a'_\gamma}{q'}
	\Big\rvert
	\geq
	\frac{1}{q q'}
	-
	\Big\lvert
	\xi_\gamma - \frac{a_\gamma}{q}
	\Big\rvert
	\geq
	2^{-s-1} N^{-\alpha}
	-
	N^{\beta - \abs{\gamma}}
	\geq
	N^{\beta -\abs{\gamma}}.
	$$
	Hence, by \eqref{eq:9} we get
	$$
	\abs{\Phi_N(\xi - a'/q')} \lesssim
	\big\lvert N^A \big(\xi - a'/q'\big) \big\rvert^{-1/d}
	\lesssim
	N^{-\beta/d}.
	$$
	In particular, by \eqref{eq:8}
	\begin{equation}
		\label{eq:69}
		\Big\lvert
		\sum_{s = 0}^{s_1-1}
		\sum_{\atop{a'/q' \in \mathscr{R}_s}{a'/q' \neq a/q}}
		G(a'/q') \Phi_N(\xi - a'/q') \eta_s(\xi - a'/q')
		\Big\rvert
		\lesssim
		N^{-\beta/d}
		\sum_{s = 0}^{s_1-1} 2^{-\delta s}.
	\end{equation}
	Next, if $\eta_{s_0}(\xi - a/q) < 1$ then $\abs{\xi_\gamma - a_\gamma/q}
	\geq 4^{-1} \cdot 10^{-(s_0+1)\abs{\gamma}}$ for some $\gamma \in \Gamma$. Since
	$2^{s_0} \leq N^{\alpha}$, by \eqref{eq:9}, we get
	\begin{equation}
		\label{eq:64}
		\big\lvert
		G(a/q) \Phi_N(\xi - a/q) (1 - \eta_{s_0}(\xi - a/q))
		\big\rvert
		\lesssim
		\big\lvert N^A \big(\xi - a/q\big) \big\rvert^{-1/d}
		\lesssim
		N^{-(1-4\alpha)/d}.
	\end{equation}
	Finally, since $\abs{\Phi_N(\xi)}$ is uniformly bounded, by \eqref{eq:8} we get
	\begin{equation}
		\label{eq:65}
		\Big\lvert
		\sum_{s = s_1}^\infty
		\sum_{\atop{a'/q' \in \mathscr{R}_s}{a'/q' \neq a/q}}
		G(a'/q') \Phi_N(\xi - a'/q') \eta_s(\xi - a'/q')
		\Big\rvert
		\lesssim
		\sum_{s = s_1}^\infty 2^{-\delta s}
		\lesssim
		N^{-\delta (1 - \alpha - \beta)}.
	\end{equation}
	Hence, by Proposition \ref{prop:1} and estimates \eqref{eq:69}, \eqref{eq:64} and \eqref{eq:65}
	there exists $\delta_1'>0$ such that for any $\xi\in \mathfrak{M}_N$
	\begin{align}
		\label{eq:74}
		\big\lvert
		m_N(\xi) - \nu_N(\xi)
		\big\rvert
		\lesssim N^{-\delta_1'}.
	\end{align}
	
	\noindent
	{\textbf{Minor arcs estimates: $\xi \in \mathfrak{m}_N$.}} By Dirichlet's principle,
	for each $\gamma \in \Gamma$ there are $1 \leq q_\gamma \leq N^{\abs{\gamma}-\beta}$ and
	$(a_\gamma, q_\gamma) = 1$ such that
	$$
	\Big\lvert \xi_\gamma - \frac{a_\gamma}{q_\gamma} \Big\rvert
	\leq N^{-\abs{\gamma} + \beta}.
	$$
	Suppose that for all $\gamma \in \Gamma$, $1 \leq q_\gamma \leq N^{\alpha/d}$. Then
	setting $q = \lcm\{q_\gamma : \gamma \in \Gamma\}$ we have $q \leq N^\alpha$. But then
	$a \in A_q$ which contradicts to $\xi \in \mathfrak{m}_N$. Therefore, there is
	$\gamma \in \Gamma$ such that $N^{\alpha/d} \leq q_\gamma \leq N^{\abs{\gamma} - \beta}$.
	By the multi-dimensional version of Weyl's inequality (see \cite{SW0}), there is $\delta' > 0$
	such that
	\begin{align}
		\label{eq:67}
		\lvert m_N(\xi) \rvert \lesssim N^{-\delta'}.
	\end{align}
	To estimate $\abs{\nu_N(\xi)}$ we define $s_1$ by setting
	$$
	2^{s_1} \leq N^{\alpha} \leq 2^{s_1 + 1}.
	$$
	If $s < s_1$ then for any $a/q \in \mathscr{R}_s$ we have $q \leq N^\alpha$ and there is
	$\gamma \in \Gamma$ such that
	$$
	\Big\lvert
	\xi_\gamma - \frac{a_\gamma}{q}
	\Big\rvert
	\geq
	N^{-\abs{\gamma} + \beta}.
	$$
	Thus, by \eqref{eq:9}
	$$
	\lvert
	\Phi_N(\xi - a/q)
	\rvert
	\lesssim
	\norm{N^A\big(\xi - a/q\big)}^{-1/d}
	\lesssim
	N^{-\beta/d}.
	$$
	Hence, by \eqref{eq:8}
	\begin{align}
		\label{eq:68}
		\Big\lvert
		\sum_{s = 0}^{s_1 - 1} \nu_N^s(\xi)
		\Big\rvert
		\lesssim
		N^{-\beta/d} \sum_{s = 0}^\infty 2^{-\delta s}
		\lesssim
		N^{-\beta / d}.
	\end{align}
	For the second part, we proceed similar to \eqref{eq:65} and obtain
	\begin{align}
		\label{eq:70}
		\Big\lvert
		\sum_{s = s_1}^\infty
		\nu_N^s(\xi)
		\Big\rvert
		\lesssim
		\sum_{s = s_1}^\infty
		2^{-\delta s}
		\lesssim
		2^{-\delta s_1}
		\lesssim
		N^{-\alpha / d}.
	\end{align}
	Combining \eqref{eq:67}, \eqref{eq:68} and \eqref{eq:70} we can find $\delta''_1 > 0$
	such that for any $\xi\in\mathfrak{m}_N$
	\begin{align}
		\label{eq:75}
		\big\lvert
		m_N(\xi) - \nu_N(\xi)
		\big\rvert
		\lesssim
		N^{-\delta_1''}.
	\end{align}
	Finally, by \eqref{eq:74} and \eqref{eq:75} taking $\delta_1=\min\{\delta_1', \delta_1''\}>0$
	we finish the proof.
\end{proof}

\subsection{$\ell^p$-theory}
Let us recall, $\eta = \phi * \psi$ for $\psi, \phi$ smooth nonnegative functions with supports
inside $[-1/2, 1/2]^d$. The next two lemmas are multi-dimensional analogues of Lemma 1 and Lemma 2
from \cite{mt2}.
\begin{lemma}
	\label{lem:2}
	For any $t > 0$ and $u \in \RR^d$
	\begin{equation}
		\label{eq:23}
		\bigg\lVert
		\int_{\TT^d}
		e^{-2\pi i \sprod{\xi}{x}} \eta\big(t^A \xi\big) d\xi
		\bigg\rVert_{\ell^1(x)} \leq 1,
	\end{equation}
	\begin{equation}
		\label{eq:24}
		\bigg\lVert
		\int_{\TT^d}
		e^{-2\pi i \sprod{\xi}{x}}
		\big(1-e^{2\pi i \sprod{\xi}{u}}\big) \eta\big(t^A \xi\big)
		d\xi
		\bigg\rVert_{\ell^1(x)}
		\leq
		\norm{t^{-A} u}.
	\end{equation}
\end{lemma}
\begin{proof}
	We only show the inequality \eqref{eq:24} since the proof of \eqref{eq:23} is almost identical.
	Let us observe that
	$$
	\eta(t^A \xi) = t^{\tr(A)} \phi_t * \psi_t(\xi)
	$$
	where $\phi_t(\xi) = \phi\big(t^A \xi\big)$, and $\psi_t(\xi) = \psi\big(t^A \xi \big)$. For
	$x \in \ZZ^d$ we have
	$$
	t^{-\tr(A)}
	\int_{\TT^d} e^{-2\pi i \sprod{\xi}{x}}
	\big(1 - e^{2\pi i \sprod{\xi}{u}}\big)
	\eta(t^A \xi) d\xi\\
	=\mathcal{F}^{-1} \phi_t (x) \mathcal{F}^{-1} \psi_t(x)
	- \mathcal{F}^{-1} \phi_t (x-u) \mathcal{F}^{-1} \psi_t(x-u).
	$$
	By Cauchy--Schwarz inequality and Plancherel's theorem
	\begin{multline*}
		\sum_{x \in \ZZ^d}
		\sabs{\mathcal{F}^{-1} \phi_t(x)}
		\sabs{\mathcal{F}^{-1} \psi_t(x)-\mathcal{F}^{-1} \psi_t(x-u)}\\
		\leq
		\vnorm{\mathcal{F}^{-1} \phi_t}_{\ell^2}
		\bigg\lVert
		\int_{\RR^d} e^{-2\pi i \sprod{\xi}{x}}
		\big(1-e^{2\pi i \sprod{\xi}{u}}\big)
		\psi_t(\xi) d\xi
		\bigg\rVert_{\ell^2(x)}
		=
		\vnorm{\phi_t}_{L^2}
		\vnorm{\big(1-e^{2\pi i \sprod{\xi}{u}}\big) \psi_t(\xi)}_{L^2(d\xi)}.
	\end{multline*}
	Moreover, since
	$$
	\int_{\RR^d} \sabs{1 - e^{-2\pi i \sprod{\xi}{u}}}^2 \abs{\psi_t(\xi)}^2 d\xi
	\lesssim \norm{t^{-A} u}^2  \int_{\RR^d} \norm{t^A \xi}^2 \abs{\psi_t(\xi)}^2 d\xi
	\lesssim  t^{-\tr(A)} \norm{t^{-A} u}^2  \int_{\RR^d} \norm{\xi}^2 \abs{\psi(\xi)}^2 d\xi
	$$
	we obtain
	$$
	\sum_{x \in \ZZ^d}
	\sabs{\mathcal{F}^{-1} \phi_t(x)}
	\sabs{\mathcal{F}^{-1} \psi_t(x) - \mathcal{F}^{-1} \psi_t(x-u)}
	\leq
	2^{-1}
	t^{-\tr(A)}
	\norm{t^{-A} u}
	\vnorm{\phi}_{L^2}
	\vnorm{\psi}_{L^2}
	$$
	which finishes the proof of \eqref{eq:24} since $\vnorm{\phi}_{L^2} \vnorm{\psi}_{L^2} \leq 1$.
\end{proof}
\begin{proposition}
	\label{prop:2}
	For each $p \in (1, \infty)$, $r > 2$ and $t > 0$ we have
	$$
	\big\lVert
	\bvar{\big(\seq{\mathcal{F}^{-1} \big(\Phi_N \eta(t^A \: \cdot) \hat{f}  \big)}{N \in \NN}\big)}
	\big\rVert_{\ell^p}
	\leq
	C_{p, r}
	\big\lVert
	\calF^{-1} \big(\eta(t^A \: \cdot) \hat{f} \big)
	\big\rVert_{\ell^p}.
	$$
\end{proposition}
\begin{proof}
	Let $\vrho_t(\xi) = \eta(t^A \xi)$. Since $\vrho_t = \vrho_t \vrho_{t/2}$. by H\"older's
	inequality we have
	\begin{multline*}
		\big\lvert
		\bvar{\big(\seq{
		\calF^{-1}
		\big(\Phi_N \vrho_t \hat{f} \big)(x)}{N \in \NN}\big)}
		\big\rvert^p
		\leq
		\bigg(\int_{\RR^d}
		\bvar{\big(\seq{\calF^{-1} \big(\Phi_N \vrho_t \hat{f} \big)(u)}{N \in \NN}\big)}
		\big\lvert
		\calF^{-1} \vrho_{t/2}(x - u)
		\big\rvert du\bigg)^p \\
		\leq
		\int_{\RR^d}
		\Big(
		\bvar{\big(\seq{\calF^{-1} \big(\Phi_N \vrho_t \hat{f} \big)(u)}{N \in \NN}\big)}
		\Big)^p
		\big\lvert \calF^{-1} \vrho_{t/2}(x - u) \big\rvert du \cdot
		\vnorm{\mathcal{F}^{-1} \vrho_{t/2}}_{L^{1}}^{p-1}.
	\end{multline*}	
	Next, we note that $\vnorm{\calF^{-1} \vrho_{t/2}}_{L^{1}}\lesssim 1$ and
	$$
	\sum_{x \in \ZZ^d} \big\lvert \calF^{-1}\vrho_{t/2} (x-u) \big\rvert
	\lesssim
	t^{-\tr(A)}
	\sum_{x \in \ZZ^d} \frac{1}{1 + \norm{t^A (x - u)}^2}
	$$
	which is uniformly bounded with respect to $A$. Thus we obtain
	$$
	\big\lVert
	\bvar{\big(\seq{\calF^{-1} \big(\Phi_N \vrho_t \hat{f} \big)}{N \in \NN}\big)}
	\big\rVert_{\ell^p}
	\lesssim
	\big\lVert
	\bvar{\seq{\big(\calF^{-1} \big(\Phi_N \vrho_t \hat{f} \big)}{N \in \NN}\big)}
	\big\rVert_{L^p}
	\lesssim
	\big\lVert \calF^{-1} \big(\vrho_t \hat{f} \big) \big\rVert_{L^p}
	$$
	where the last inequality is a consequence of Remark after Theorem 1.5 in \cite[p. 6717]{JSW}.
	The proof will be completed if we show
	$$
	\big\lVert
	\calF^{-1} \big(\vrho_t \hat{f}\big) \big\rVert_{L^p}
	\lesssim
	\big\lVert \calF^{-1} \big(\vrho_t \hat{f}\big) \big\rVert_{\ell^p}.
	$$
	For this purpose we use \eqref{eq:24} from Lemma \ref{lem:2}. We have
	\begin{multline*}
		\sum_{x \in \ZZ^d}
		\int_{[0, 1]^d}
		\big\lvert
		\calF^{-1}\big(\vrho_t \hat{f} \big)(x+u) - \calF^{-1} \big(\vrho_t \hat{f} \big)(x)
		\big\rvert^p
		du\\
		\leq
		\int_{[0, 1]^d}
		\Big\lVert
		\int_{\TT^d}
		e^{-2\pi i \sprod{\xi}{x}} \big(1 - e^{-2\pi i \sprod{\xi}{u}}\big)
		\vrho_{t/2}(\xi)
		d\xi
		\Big\rVert_{\ell^1(x)}^p
		\big\lVert \calF^{-1}\big(\vrho_t \hat{f}\big)\big\rVert_{\ell^p}^p
		du
		\lesssim
		\big\lVert
		\calF^{-1}\big(\vrho_t \hat{f} \big)
		\big\rVert_{\ell^p}^p.
	\end{multline*}
	Hence,
	$$
	\big\lVert \calF^{-1} \big(\vrho_t \hat{f}\big)\big\rVert_{L^p}^p
	=
	\sum_{x \in \ZZ^d}
	\int_{[0,1]^d}
	\big\lvert \calF^{-1} \big(\vrho_t \hat{f} \big)(x+u)\big\rvert^p
	du
	\lesssim
	\big\lVert
	\calF^{-1}\big(\vrho_t \hat{f} \big)
	\big\rVert_{\ell^p}^p.
	$$
	This finishes the proof.
\end{proof}

For $t \in \NN_0$ we set $Q_t = \big(2^{t+1}\big)!$ and define
$$
\vrho_t(\xi) = \eta\big(Q_{t+1}^{3dA} \xi\big).
$$
\begin{lemma}
	\label{lem:3}
	Let $p \in [1, \infty)$ and $m \in \NN_{Q_t}^d$. Then for any $t\in\NN_0$ we have
	$$
	\big\lVert
	\mathcal{F}^{-1} \big(\vrho_t \hat{f}\big)(Q_t x + m)
	\big\rVert_{\ell^{p}(x)}
	\simeq
	Q_t^{-d/p}
	\big\lVert
	\mathcal{F}^{-1} \big(\vrho_t \hat{f}\big)
	\big\rVert_{\ell^{p}}.
	$$
\end{lemma}
\begin{proof}
	For each $m \in \NN_{Q_t}^d$ we set
	$$
	J_m =
	\big\lVert \mathcal{F}^{-1} \big(\vrho_t \hat{f} \big)(Q_t x+m) \big\rVert_{\ell^p(x)},
	$$
	and $I = \big\lVert \mathcal{F}^{-1} \big(\vrho_t \hat{f} \big) \big\rVert_{\ell^p}$.
	Then
	$$
	\sum_{m \in \NN_{Q_t}^d} J_m^p = I^p.
	$$
	Since $\vrho_t = \vrho_t \vrho_{t-1}$, by Minkowski's inequality we obtain that
	\begin{multline*}
		\big\lVert
		\mathcal{F}^{-1} \big(\vrho_t \hat{f} \big)(Q_t x + m)
		-\mathcal{F}^{-1}\big(\vrho_t \hat{f} \big)(Q_t x + m')
		\big\rVert_{\ell^p(x)}\\
		=
		\bigg\lVert
		\int_{\TT^d} e^{-2\pi i \sprod{\xi}{(Q_t x + m)}}
		\big(1 - e^{2\pi i \sprod{\xi}{(m-m')}}\big)
		\vrho_t(\xi) \hat{f}(\xi) d\xi
		\bigg\rVert_{\ell^p(x)}\\
		\leq
		\bigg\lVert
		\int_{\TT^d} e^{-2\pi i \sprod{\xi}{x}}
		\big(1 - e^{2\pi i \sprod{\xi}{(m-m')}}\big)
		\vrho_{t-1}(\xi) d\xi
		\bigg\rVert_{\ell^1(x)} I
		\leq \big|Q_t^{-3dA}(m-m')\big| I
	\end{multline*}	
	where in the last step we have used Lemma \ref{lem:2}. Hence, for all $m, m' \in N_{Q_t}^d$
	$$
	J_{m'} \leq J_m + Q_t^{1-3d} I \leq J_m + Q_t^{-2d} I.
	$$
	Thus
	\begin{equation}
		\label{eq:4}
		J_{m'}^p \leq 2^{p-1} J_m^p + 2^{p-1} Q_t^{-2d p} I^p.
	\end{equation}
	Therefore,
	$$
	I^p = \sum_{m' \in \NN_{Q_t}^d} J_{m'}^p
	\leq 2^{p-1} Q_t^d J_m^p + 2^{p-1} Q_t^{d(1-2p)} I^p.
	$$
	By the definition of $Q_t$ we have
	$$
	2^p Q_t^{d(1-2p)} \leq 1.
	$$
	Hence, we obtain $I^p \leq 2^p Q_t^d J_m^p$. For the converse inequality, we use again
	\eqref{eq:4} to get
	$$
	Q_t^d J_{m'}^p \leq 2^{p-1} \sum_{m \in \NN_{Q_t}^d} J_m^p
	+ 2^{p-1} Q_t^{d(1-2p)} I^p \leq 2^p I^p.
	$$
\end{proof}

\begin{proposition}
	\label{prop:4}
	Let $p \in (1, \infty)$, $r > 2$ and  $m\in\NN_{Q_t}^k$. Then for any $t\in\NN_0$
	we have
	\begin{equation*}
		\big\lVert
		\bvar{\big(\seq{\calF^{-1}\big(\Phi_N \vrho_t \hat{f}\big)(Q_t x + m)}{N \in \NN}\big)}
		\big\rVert_{\ell^p(x)}
		\lesssim
		\big\lVert
		\calF^{-1}\big(\vrho_t \hat{f} \big)(Q_t x + m)
		\big\rVert_{\ell^p(x)}.
	\end{equation*}
\end{proposition}
\begin{proof}
	For each $m \in \NN_{Q_t}^d$ we define
	$$
	J_m =
	\big\lVert
	\bvar{\big(\seq{\calF^{-1}\big(\Phi_N \vrho_t \hat{f} \big)(Q_t x + m)}{N \in \NN}\big)}
	\big\rVert_{\ell^p(x)}.
	$$
	Then, by Proposition \ref{prop:2},
	$$
	I^p =
	\sum_{m \in \NN_{Q_t}^d} J_m^p =
	\big\lVert
	\bvar{\big(\seq{\calF^{-1}\big(\Phi_N \vrho_t \hat{f} \big)}{N \in \NN}\big)}
	\big\rVert_{\ell^p}^p
	\lesssim
	\big\lVert
	\calF^{-1}\big(\vrho_t \hat{f} \big)
	\big\rVert_{\ell^p}^p.
	$$
	If $m, m' \in \NN_{Q_t}^d$ then we may write
	\begin{multline*}
		\bigg\lVert
		\bvar{\bigg(
			\seq{
				\int_{\TT^d}
				e^{-2\pi i \sprod{\xi}{(Q_t x + m)}}
				\big(1 - e^{2\pi i \sprod{\xi}{(m -m')}}\big)
			\Phi_N(\xi) \vrho_t(\xi) \hat{f}(\xi) d\xi}{N \in \NN}
		\bigg)}
		\bigg\rVert_{\ell^p(x)}\\
		\lesssim
		\Big\lVert
		\int_{\TT^d}
		e^{-2\pi i \sprod{\xi}{x}}
		\big(1  - e^{-2\pi i \sprod{\xi}{(m-m')}}\big)
		\vrho_t(\xi) \hat{f}(\xi)
		\Big\rVert_{\ell^p(x)}.
	\end{multline*}
	Since $\vrho_t = \vrho_t \vrho_{t-1}$, by Minkowski's inequality and Lemma \ref{lem:2}
	the last expression may be dominated by
	$$
	\Big\lVert
	\int_{\TT^d}
	e^{-2\pi i \sprod{\xi}{x}}
	\big(1  - e^{-2\pi i \sprod{\xi}{(m-m')}}\big)
	\vrho_{t-1}(\xi) \hat{f}(\xi)
	\Big\rVert_{\ell^1(x)}
	\big\lVert
	\calF^{-1}\big(\vrho_t \hat{f} \big)
	\big\rVert_{\ell^p}
	\leq
	Q_t^{-2d}
	\big\lVert
	\calF^{-1}\big(\vrho_t \hat{f}\big)
	\big\rVert_{\ell^p},
	$$
	thus
	$$
	J_m \leq J_{m'} + Q_t^{-2d}
	\big\lVert
	\calF^{-1}\big(\vrho_t \hat{f} \big)
	\big\rVert_{\ell^p}.
	$$
	Raising to $p$'th power and summing up over $m' \in \NN_{Q_t}^d$ we get
	$$
	Q_t^d J_m^p \leq 2^{p-1} I^p
	+ 2^{p-1} Q_t^{d(1-2p)}
	\big\lVert
	\calF^{-1}\big(\vrho_t \hat{f} \big)
	\big\rVert_{\ell^p}^p
	\lesssim
	\big\lVert
	\calF^{-1}\big(\vrho_t \hat{f} \big)
	\big\rVert_{\ell^p}^p
	$$
	and Lemma \ref{lem:3} finishes the proof.
\end{proof}
\subsection{Unrestricted inequalities}
We start by proving $\ell^2(\ZZ^d)$-boundedness of $r$-variations for $\nu_{2^j}$. The proofs
of the estimates like in \eqref{eq:63} are based on Bourgain's `logarithmic' type lemmas
(see \cite{bou}, see also \cite{K, K2}). We present different approach, based on a direct analysis
of this multiplier. For the proof of Theorem \ref{thm:1} in case $p = 2$ we use Theorem
\ref{thm:6}.

\begin{theorem}
	\label{thm:6}
	For any $r>2$, there are $\delta_2>0$ and $C>0$ such that for any $s \in \NN_0$ and
	$f\in\ell^2(\ZZ^d)$
	\begin{align}
		\label{eq:63}
		\big\lVert
		\var{\big(\seq{\calF^{-1}\big(\nu_{2^j}^s \hat{f} \big)}
			{j\ge 0}\big)}
		\big\rVert_{\ell^2}
		\leq
		C
		2^{-s\delta_2 }
		\vnorm{f}_{\ell^2}.
	\end{align}
\end{theorem}

\begin{proof}
	The proof will consist of two parts where we shall estimate separately the pieces of
	$r$-variations where $0\le j \le 2^{\kappa_s}$ and $j \ge 2^{\kappa_s}$ where
	$\kappa_s = 20 d(s+1)$. By \eqref{eq:19} and \eqref{eq:25} we see that
	\begin{multline}
		\label{eq:76}
		\big\lVert
		\var{\big(\seq{\calF^{-1}\big(\nu_{2^j}^s \hat{f} \big)} {j\ge 0}\big)}
		\big\rVert_{\ell^2}
		\lesssim \big\lVert
		\calF^{-1}\big(\nu_{2^{2^{\kappa_s}}}^s \hat{f} \big)
		\big\rVert_{\ell^2}
		+\big\lVert
		\var{\big(\seq{\calF^{-1}\big(\nu_{2^j}^s \hat{f} \big)}
			{0\le j\le 2^{\kappa_s}}\big)}
		\big\rVert_{\ell^2}\\
		+\big\lVert
		\var{\big(\seq{\calF^{-1}\big(\nu_{2^j}^s \hat{f} \big)}
			{j\ge 2^{\kappa_s}}\big)}
		\big\rVert_{\ell^2}.
	\end{multline}
	By Plancherel's theorem, \eqref{eq:8} and the disjointness of supports of $\eta_s(\xi-a/q)$'s
	while $a/q$ varies over $\mathscr{R}_s$, the first term in \eqref{eq:76} is bounded by
	$2^{-s\delta}\|f\|_{\ell^2}$. To estimate the second term we apply Lemma \ref{lem:6}. Indeed,
	let $I^i_j=[j2^i, (j+1)2^i)$ and note that \eqref{eq:62} and Plancherel's theorem give
	\begin{multline*}
		\big\lVert
		\var{\big(\seq{\calF^{-1}\big(\nu_{2^j}^s \hat{f} \big)}
			{0\le j\le 2^{\kappa_s}}\big)}
		\big\rVert_{\ell^2}
		\lesssim
		\sum_{i=0}^{\kappa_s}
		\bigg\|\bigg(
		\sum_{j=0}^{2^{\kappa_s-i}-1}
		\Big|\sum_{m \in I_j^i}
		\calF^{-1}\big(\nu_{2^{m+1}}^s\hat{f} \big)
		-\calF^{-1}\big(\nu_{2^m}^s \hat{f} \big)
		\Big|^2\bigg)^{1/2}
		\bigg\|_{\ell^2}\\
		=
		\sum_{i=0}^{\kappa_s}
		\bigg(\sum_{j=0}^{2^{\kappa_s-i}-1} \int_{\TT^d}
		\Big|
		\sum_{m \in I^i_j}
		\big(\nu_{2^{m+1}}^s(\xi)-\nu_{2^m}^s(\xi)\big) \hat{f}(\xi)
		\Big|^2
		d\xi
		\bigg)^{1/2}.
	\end{multline*}
	Next, for any $i \in \{0, 1, \ldots, \kappa_s\}$ we have	
	\begin{multline*}
		\sum_{j=0}^{2^{\kappa_s-i}-1} \int_{\TT^d}
		\Big|
		\sum_{m \in I^i_j}
		\nu_{2^{m+1}}^s(\xi)-\nu_{2^m}^s(\xi)
		\Big|^2
		\big\lvert \hat{f}(\xi) \big\rvert^2
		d\xi\\
		\le
		\sum_{j=0}^{2^{\kappa_s-i}-1}
		\int_{\TT^d}
		\sum_{m, m' \in I^i_j}
		\big|\nu_{2^{m+1}}^s(\xi)-\nu_{2^m}^s(\xi) \big|
		\cdot
		\big|\nu_{2^{m'+1}}^s(\xi)-\nu_{2^{m'}}^s(\xi) \big|
		\cdot
		\big|\hat{f}(\xi)\big|^2
		d\xi.
	\end{multline*}
	Let $\Delta_m(\xi) = \Phi_{2^{m+1}}(\xi) - \Phi_{2^m}(\xi)$. Using \eqref{eq:9} and
	\eqref{eq:10} we can estimate
	\begin{align*}
		\sum_{m \in \NN}
		\abs{\Delta_m(\xi)}
		\lesssim
		\sum_{m \in \NN}
		\min\big\{|2^{mA}\xi|, |2^{mA}\xi|^{-1/d}\big\}
		\lesssim 1.
	\end{align*}
	Therefore, by the disjointness of supports of $\eta_s(\cdot-a/q)$'s we obtain
	\begin{multline*}
		\sum_{j = 0}^{2^{\kappa_s-i}-1}
		\int_{\TT^d}
		\Big|
		\sum_{m \in I^i_j}
		\nu_{2^{m+1}}^s(\xi)-\nu_{2^m}^s(\xi)
		\Big|^2
		\big\lvert \hat{f}(\xi) \big\rvert^2
		d\xi\\
		\le
		\sum_{a/q \in \mathscr{R}_s}
		|G(a/q)|^2
		\sum_{j = 0}^{2^{\kappa_s-i}-1}
		\sum_{m, m' \in I_j^i}
		\int_{\TT^d}
		\abs{\Delta_m(\xi- a/q)}
		\cdot
		\abs{\Delta_{m'}(\xi - a/q)}
		\cdot
		\eta_s(\xi-a/q)^2
		\big| \hat{f}(\xi)\big|^2
		d\xi\\
		\lesssim
		\sum_{a/q \in \mathscr{R}_s}
		\abs{G(a/q)}^2
		\int_{\TT^d}
		\bigg(
		\sum_{j=0}^{\infty}\min\big\{|2^{jA}(\xi-a/q)|, |2^{jA}(\xi-a/q)|^{-1/d}\big\}
		\bigg)^2
		\eta_s(\xi-a/q)^2
		\big| \hat{f}(\xi)\big|^2
		d\xi
	\end{multline*}
	which, by \eqref{eq:8}, is bounded by $2^{-2 s \delta} \lVert f \rVert_{\ell^2}^2$. Finally,
	it remains to estimate the last term in \eqref{eq:76}. Let us observe that if $x \in \ZZ^d$
	then
	$$
	\calF^{-1}\big(\nu_{2^j}^s \hat{f} \big)(x)
	=
	\sum_{a/q \in \mathscr{R}_s}
	e^{-2\pi i \sprod{(a/q)}{x}}
	G(a/q)
	\calF^{-1}\big(\Phi_{2^j} \eta_s \hat{f}(\cdot + a/q)\big)(x).
	$$
	For any $x, y \in \ZZ^d$ we set
	\begin{align*}
		I(x, y) =
		\var{\Big(
		\seq{\sum_{a/q \in \mathscr{R}_s} G(a/q) e^{-2\pi i\sprod{(a/q)}{x}}
		\calF^{-1}\big(\Phi_{2^j} \eta_s \hat{f}(\cdot + a/q)\big)(y)
		}{j \geq 2^{\kappa_s}}\Big)}.
	\end{align*}
	and
	\begin{align*}
		J(x, y) =
		\sum_{a/q \in \mathscr{R}_s} G(a/q) e^{-2\pi i \sprod{(a/q)}{x}}
		\calF^{-1} \big( \eta_s \hat{f}(\cdot + a/q)\big)(y).
	\end{align*}
	We notice, functions $x \mapsto I(x, y)$ and $x \mapsto J(x, y)$ are $Q_s \ZZ^d$-periodic.
	If $u \in \NN^d_{Q_s}$ and $a/q \in \mathscr{R}_s$, by Plancherel's theorem we get
	\begin{multline*}
		\big\lVert
		\calF^{-1}\big(\Phi_{2^j} \eta_s \hat{f}(\cdot + a/q)\big)(x+u)
		-
		\calF^{-1}\big(\Phi_{2^j} \eta_s \hat{f}(\cdot + a/q)\big)(x)
		\big\rVert_{\ell^2(x)}\\
		=
		\big\lVert
		\big(1 - e^{- 2\pi i \sprod{\xi}{u}} \big)
		\Phi_{2^j}(\xi)
		\eta_s(\xi)
		\hat{f}(\xi + a/q)
		\big\rVert_{L^2(d\xi)}
		\lesssim
		2^{-j/d} \norm{u} \cdot
		\big\lVert
		\eta_s(\cdot - a/q) \hat{f}
		\big\rVert_{L^2}
	\end{multline*}
	since, for $\xi \in \TT^d$
	$$
	\norm{\xi} \abs{\Phi_{2^j}(\xi)} \lesssim \norm{\xi} \norm{2^{jA} \xi}^{-1/d}
	\lesssim 2^{-j/d}.
	$$
	Therefore, by the triangle inequality, \eqref{eq:18} and Plancherel's theorem
	$$
	\big\lvert
	\lVert
	I(x, x+u)
	\rVert_{\ell^2(x)}
	-
	\lVert
	I(x, x)
	\rVert_{\ell^2(x)}
	\big\rvert
	\leq
	Q_s
	\sum_{j\ge 2^{\kappa_s}}
	\sum_{a/q\in\mathscr{R}_s}
	2^{-j/d}
	\big\lVert \eta_s(\cdot - a/q) \hat{f}
	\big\rVert_{L^2}.
	$$
	Since $\mathscr{R}_s$ contains at most $2^{s(d+1)}$ rational numbers we have
	$$
	\sum_{a/q \in \mathscr{R}_s}
	\big\lVert
	\eta_s(\cdot - a/q) \hat{f}
	\big\rVert_{L^2}
	\leq
	2^{(d+1)s}
	\vnorm{f}_{\ell^2}.
	$$
	Hence, using $2^{\kappa_s}/d - (s+1)2^{s+1} - (d+1) s \geq \delta s$ we obtain
	$$
	\lVert I(x, x) \rVert_{\ell^2(x)}
	\lesssim
	\lVert I(x, x+u) \rVert_{\ell^2(x)}
	+
	2^{-\delta s}
	\lVert f \rVert_{\ell^2}.
	$$
	Thus, we may estimate
	\begin{align}
		\label{eq:33}
		\Big\lVert
		\var{\big(\seq{\calF^{-1}\big(\nu_j^s \hat{f}\big)}{j \geq 2^{\kappa_s}}\big)}
		\Big\rVert_{\ell^2}^2
		\lesssim
		\frac{1}{Q_s^d}
		\sum_{u\in\NN_{Q_s}^d}
		\lVert I(x, x+u) \rVert_{\ell^2(x)}^2
		+
		2^{-2 \delta s} \vnorm{f}_{\ell^2}^2.
	\end{align}
	Next, by double change of variables and periodicity we get
	$$
	\sum_{u \in \NN_{Q_s}^d}
	\lVert I(x, x+u) \rVert_{\ell^2(x)}^2
	=
	\sum_{x \in \ZZ^d}
	\sum_{u \in \NN_{Q_s}^d}
	I(x-u, x)^2
	=
	\sum_{x \in \ZZ^d}
	\sum_{u \in \NN_{Q_s}^d} I(u, x)^2
	=
	\sum_{u \in \NN_{Q_s}^d}
	\lVert I(u, x) \rVert_{\ell^2(x)}^2
	$$
	what, using Proposition \ref{prop:2} and \eqref{eq:8}, is bounded by
	\begin{align*}
		\sum_{u \in \NN_{Q_s}^d}
		\lVert
		J(u, x)
		\rVert_{\ell^2(x)}^2
		=
        \sum_{u \in \NN_{Q_s}^d}
	   \lVert J(x, x+u) \rVert_{\ell^2(x)}^2
        =
		\sum_{u \in \NN_{Q_s}^d}
		\int_{\TT^d}\Big|\sum_{a/q\in\mathscr{R}_s}G(a/q) e^{2\pi i \sprod{(a/q)}{u}}
		\eta_s(\xi-a/q)
		\hat{f}(\xi)\Big|^2
		d\xi\\
		\lesssim
		2^{-2\delta s}
		Q_s^d
		\cdot
		\|f\|_{\ell^2}^2.
	\end{align*}
	Finally, combining with \eqref{eq:33} we obtain an estimate on the last term in
	\eqref{eq:76}.
\end{proof}

\begin{theorem}\label{thm:7}
	There exists $C>0$ such that for every $f\in\ell^2\big(\ZZ^d\big)$
	$$
	\big\lVert
	\sup_{N \in \NN}
	\big\lvert M_N f \big\rvert
	\big\rVert_{\ell^2}
	\leq
	C \vnorm{f}_{\ell^2}.
	$$
\end{theorem}
\begin{proof}
	In view of \eqref{eq:25} it suffices to apply Theorem \ref{thm:6} and Lemma \ref{lem:4}.
\end{proof}
For each $N \in \NN$ and $t\in\NN_0$ we define  new multipliers
\begin{align}
	\label{eq:71}
	\Omega_N^t(\xi)
	=\sum_{a \in \NN_{Q_t}^d}
	G(a/Q_t) \Phi_N(\xi - a/Q_t) \vrho_t(\xi - a/Q_t).
\end{align}
Then
\begin{theorem}
	\label{th:1}
	Let $p \in (1, \infty)$ and $r > 2$. There exists $C_{p, r}>0$ such that for any
	$t \in\NN_0$
	$$
	\big\lVert
	\bvar{\big(\seq{\calF^{-1}\big(\Omega_N^t \hat{f}\big)}{N \in \NN}\big)}
	\big\rVert_{\ell^p}
	\leq
	C_{r, p}
	\vnorm{f}_{\ell^p}.
	$$
\end{theorem}
\begin{proof}
	Let us observe that
	$$
	\calF^{-1}\big(\Phi_N (\cdot - a/Q_t) \vrho_t(\cdot - a/Q_t) \hat{f} \big)(Q_t x + m)
	=
	\calF^{-1}\big(\Phi_N \vrho_t \hat{f}(\cdot  + a/Q_t)\big)(Q_t x + m) e^{-2\pi i \sprod{(a/Q_t)}{m}}.
	$$
	Therefore,
	$$
	\big\lVert
	\bvar{\big(\seq{\calF^{-1}\big(\Omega_N^t \hat{f} \big)}{N \in \NN}\big)}
	\big\rVert_{\ell^p}^p
	=
	\sum_{m \in \NN_{Q_t}^d}
	\big\lVert
	\bvar{\big(\seq{\calF^{-1}\big(\Phi_N \vrho_t F(\cdot\ ; m) \big)(Q_t x + m)} {N \in \NN} \big)}
	\big\rVert_{\ell^p(x)}^p
	$$
	where
	\begin{equation}
		\label{eq:21}
		F(\xi; m)
		= \sum_{a \in \NN_{Q_t}^d} G(a/Q_t) \hat{f}(\xi + a/Q_t) e^{-2\pi i \sprod{(a/Q_t)}{m}}.
	\end{equation}
	Now, by Proposition \ref{prop:4} and the definition \eqref{eq:21} we get
	\begin{multline*}
		\sum_{m \in \NN_{Q_t}^d}
		\big\lVert
		\bvar{\big(\seq{\calF^{-1}\big(\Phi_N \vrho_t F(\cdot\ ; m) \big)(Q_t x + m)}
		{N \in \NN} \big)} \big\rVert_{\ell^p(x)}^p\\
		\lesssim
		\sum_{m \in \NN_{Q_t}^d}
		\big\lVert
		\calF^{-1}\big( \vrho_t F(\cdot\ ; m)\big)(Q_t x + m)
		\big\rVert_{\ell^p(x)}^p
		=
		\Big\lVert
		\sum_{a \in \NN_{Q_t}^d} G(a/Q_t) \calF^{-1}\big(\vrho_t(\cdot - a/Q_t) \hat{f}\big)
		\Big\rVert_{\ell^p}^p.
	\end{multline*}
	Using Minkowski's inequality we may estimate
	$$
	\Big\lVert
	\sum_{a \in \NN_{Q_t}^d} G(a/Q_t) \calF^{-1}\big(\vrho_t(\cdot - a/Q_t) \hat{f}\big)
	\Big\rVert_{\ell^p}
	\leq
	\Big\lVert
	\sum_{a \in \NN_{Q_t}^d} e^{-2\pi i \sprod{(a/Q_t)}{x}} G(a/Q_t) \calF^{-1} \vrho_t(x)
	\Big\rVert_{\ell^1(x)}
	\vnorm{f}_{\ell^p}.
	$$
	We notice that for $x \in \ZZ^d$ we have
	$$
	\sum_{a \in \NN_{Q_t}^d} e^{-2\pi i \sprod{(a/Q_t)}{x}}
	=
	\begin{cases}
		Q_t^d & \text{ if } Q_t \mid x, \\
		0 & \text{ otherwise.}
	\end{cases}
	$$
	Thus, if $x \equiv m \pmod {Q_t}$ then
	$$
	\sum_{a \in \NN_{Q^t}^d} e^{-2\pi \sprod{(a/Q_t)}{x}} G(a/Q_t)
	=
	Q_t^{-k}
	\sum_{y \in \NN_{Q_t}^k} \sum_{a \in \NN_{Q_t}^d}
	e^{2\pi i \sprod{(a/Q_t)}{(\calQ(y)-m)}}
	=Q_t^{d-k} L_m
	$$
	where $L_m = \#\big\{y \in \NN_{Q_t}^k : \calQ(y) \equiv m \pmod {Q_t}\big\}$. Let us observe
	that $L_m\cap L_{m'}=\emptyset$ if $m\not=m'$. Now, by Lemma \ref{lem:3}
	\begin{multline*}
		\Big\lVert
		\sum_{a \in \NN_{Q_t}^d} e^{-2\pi i \sprod{(a/Q_t)}{x}} G(a/Q_t)
		\calF^{-1}\vrho_t (x)
		\Big\rVert_{\ell^1(x)}
		=
		Q_t^{d-k}
		\sum_{m \in \NN_{Q_t}^d}
		L_m \cdot
		\big\lVert
		\calF^{-1}\vrho_t (Q_t x + m)
		\big\rVert_{\ell^1(x)}\\
		\lesssim
		Q_t^{-k} \sum_{m \in \NN_{Q_t}^d} L_m \cdot
		\big\lVert
		\calF^{-1}\vrho_t
		\big\rVert_{\ell^1}\lesssim1
	\end{multline*}
	what together with Lemma \ref{lem:2} finishes the proof.
\end{proof}

\subsection{Restricted inequalities}
This subsection is devoted to study certain multipliers with $r$-variations restricted to large
and small cubes, i.e. when the side length of cubes in our averages is small or large.
Let us define
\begin{align}
	\label{eq:77}
	\kappa_t = 20 d(t+1).
\end{align}

\subsubsection{Large cubes}
For any $t \in \NN_0$ we will consider a variational norm for averages over cubes with
sides bigger that $2^{2^{\kappa_t}}$. First, let us define auxiliary multipliers for each
$N \in \NN$ and $t\in\NN_0$ by
\begin{align}
	\label{eq:72}
	\Lambda_N^t(\xi)
	= \sum_{a/q \in \mathscr{Q}_t}
	G(a/q) \Phi_N(\xi - a/q) \vrho_t(\xi - a/q)
\end{align}
where
$$
\mathscr{Q}_t = \big\{a/q \in \QQ^d : q \geq 2^{t+1},\ q \mid Q_t \mbox{ and } a \in A_q\big\}.
$$
We show
\begin{proposition}
	\label{prop:3}
	Let $r>2$. There are $\delta_3>0$ and $C_r>0$ such that for any $t \in \NN_0$ and
	$f \in \ell^2\big(\ZZ^d\big)$
	\begin{align*}
		\big\lVert
		\bvar{\big(\seq{\calF^{-1}\big(\Lambda_{2^j}^t \hat{f}\big)}{j \geq 2^{\kappa_t}\big)}}
		\big\rVert_{\ell^2}
		\leq
		C_r
		2^{-t\delta_3}
		\vnorm{f}_{\ell^2}.
	\end{align*}
\end{proposition}
\begin{proof}
	We notice that
	$$
	\calF^{-1}\big(\Lambda_{2^j}^t \hat{f}\big)(Q_tx+m)
	= \calF^{-1}\big(\Phi_{2^j} \vrho_t F(\cdot\ ; m)\big)(Q_t x + m)
	$$
	where
	$$
	F(\xi; m) = \sum_{a/q \in \mathscr{Q}_t} G(a/q) \hat{f}(\xi + a/q) e^{-2\pi i \sprod{(a/q)}{m}}.
	$$
	By Proposition \ref{prop:4} and Lemma \ref{lem:3} we get
	\begin{multline*}
		\Big\lVert
		\bvar{\big(\seq{\calF^{-1}\big(\Lambda_{2^j}^t \hat{f}\big)}{j \geq 2^{\kappa_t}}\big)}
		\Big\rVert_{\ell^2}^2
		=
		\sum_{m \in \NN^d_{Q_t}}\Big\lVert
		\bvar{\big(\seq{\calF^{-1}
		\big(\Lambda_{2^j}^t \hat{f}\big)(Q_tx+m)}{j \geq 2^{\kappa_t}}\big)}
		\Big\rVert_{\ell^2(x)}^2\\
		\lesssim
		\sum_{m \in \NN^d_{Q_t}}
		\big\lVert
		\calF^{-1}\big(\vrho_t F(\cdot\ ; m) \big)(Q_tx+m)
		\big\rVert_{\ell^2(x)}^2
		=
		\Big\lVert
		\sum_{a/q \in \mathscr{Q}_t}
		G(a/q)
		\calF^{-1}\big(\vrho_t(\cdot - a/q) \hat{f} \big)
		\Big\rVert_{\ell^2}^2.
	\end{multline*}
	Using Plancherel's theorem we may write
	$$
	\Big\lVert
	\sum_{a/q \in \mathscr{Q}_t}
	G(a/q)
	\calF^{-1}\big(\vrho_t(\cdot - a/q) \hat{f}\big)
	\Big\rVert_{\ell^2}^2
	=
	\sum_{a/q \in \mathscr{Q}_t}
	\abs{G(a/q)}^2
	\big\lVert
	\vrho_t(\cdot - a/q) \hat{f}
	\big\rVert_{L^2}^2
	$$
	which, by \eqref{eq:8}, is bounded by $2^{- 2 \delta t} \vnorm{f}_{\ell^2}^2$.
\end{proof}
Finally, we show
\begin{theorem}
	\label{th:4}
	Let $r>2$. There are $\delta_4 > 0$ and $C_r > 0$ such that for any $t \in \NN_0$
	and $f \in \ell^2\big(\ZZ^d\big)$
	$$
	\Big\lVert
	\bvar{\big(\seq{M_{2^j} f - \calF^{-1}\big(\Omega_{2^j}^t \hat{f}\big)}
	{j > 2^{\kappa_t}}\big)}
	\Big\rVert_{\ell^2}
	\leq
	C_r
	2^{-\delta_4 t}
	\vnorm{f}_{\ell^2}.
	$$
\end{theorem}
\begin{proof}
	Let us notice
	$$
	\Omega_{2^j}^t(\xi)
	=
	\sum_{s = 0}^t \sum_{a/q \in \mathscr{R}_s} G(a/q) \Phi_{2^j}(\xi - a/q) \vrho_t(\xi - a/q)
	+
	\Lambda_{2^j}^t(\xi)
	$$
	and observe that
	\begin{align*}
		m_{2^j}(\xi)-\Omega_{2^j}^t(\xi)
		=\Big(m_{2^j}(\xi) - \sum_{s\ge0}\nu_{2^j}^s(\xi)\Big)
		+\Big(\sum_{s = 0}^t \nu_{2^j}^s(\xi)-\Omega_{2^j}^t(\xi)-\Lambda_{2^j}^t(\xi)\Big)
		+\sum_{s>t}\nu_{2^j}^s(\xi)+\Lambda_{2^j}^t(\xi).
	\end{align*}
	The last two terms are covered by Theorem \ref{thm:6} and Proposition \ref{prop:3}
	respectively, whereas the first term is bounded thanks to Lemma \ref{lem:4} since
	$j \geq 2^{\kappa_t}$. Thus it remains to estimate the second term. First, we observe that
	$\vrho_t(\xi - a/q) - \eta_s(\xi - a/q) \neq 0$ implies that there is $\gamma \in \Gamma$
	such that $\abs{\xi_\gamma - a_\gamma/q} \geq 4^{-1} Q_t^{-3 d\abs{\gamma}}$.
	Therefore, for $j \geq 2^{\kappa_t}$
	$$
	2^{j\abs{\gamma}} \cdot \Big\lvert \xi_{\gamma} - \frac{a_{\gamma}}{q} \Big\rvert
	\gtrsim
	2^j Q_t^{-3d^2} \gtrsim 2^{j/2},
	$$
	and using \eqref{eq:9}, we get
	$$
	\sabs{\Phi_{2^j}(\xi - a/q)} \lesssim 2^{-j/(2d)}.
	$$
	Finally, by \eqref{eq:8} we obtain
	$$
	\Big\lvert
	\Omega_{2^j}^t(\xi) - \Lambda_{2^j}^t(\xi) - \sum_{s = 0}^t \nu_{2^j}^s(\xi)
	\Big\rvert
	\leq
	\sum_{s = 0}^t 2^{-\delta s} \sabs{\Phi_{2^j}(\xi-a/q)}
	\cdot
	\sabs{\vrho_t(\xi-a/q) - \eta_s(\xi - a/q)}
	\lesssim
	2^{-j/(2d)}.
	$$
	This completes the proof of Theorem \ref{th:4}.
\end{proof}

\subsubsection{Small cubes}
Theorem \ref{th:5} will be the main result of this subsection. The proof will based on ideas
of Bourgain \cite{bou}. Bourgain used this restricted type maximal function with logarithmic loss
to obtain the full range of $p \in (1, \infty)$ in his maximal theorem.

\begin{theorem}
	\label{th:5}
	For every $p \in (1, \8]$ there exists a constant $C_p>0$ such that for all $J\in \NN$
	\begin{align}
		\label{eq:41}
		\big\|\sup_{J<j\le 2J}\big|K_{2^{j}}*f\big|\big\|_{\ell^p}
		\leq
		C_p (\log J)
		\|f\|_{\ell^p},
	\end{align}
	for every $f\in\ell^p(\ZZ^d)$.
\end{theorem}
\begin{proof}
	Since we are working with the averaging operator it suffices to prove \eqref{eq:41}
	for $p \in (1, 2]$ and nonegative function $f$. Let $\widetilde{K}_{m}(x)=K_{m}(-x)$. By the
	duality, for every $x \in \ZZ^d$, there is a sequence of nonnegative numbers
	$\big(\seq{g_j(x)}{J < j \leq 2J}\big)$ such that $\sum_{J < j \leq 2J} g_j(x) \le 1$ and
	$$
	\big\| \sup_{J <j \le 2J} |K_{2^j}*f| \big\|_{\ell^p}
	\leq
	2\Big\| \sum_{J<j\le 2J}(K_{2^j}*f)g_j \Big\|_{\ell^p}
	\le
	2\sup_{\|h\|_{\ell^r}\le1}
	\Big\| \sum_{J<j\le 2J}\widetilde{K}_{2^j}*(hg_{j}) \Big\|_{\ell^{r}}
	\|f\|_{\ell^p},
	$$
	where $r=p/(p-1) \ge 2$. Therefore, it suffices to prove that for every $p \in (1, 2]$ with an
	integer $r=p/(p-1)$ and any finite $F\subseteq\ZZ^d$ we have
	$$
	\Big\|
	\sum_{J<j\le 2J}
	\widetilde{K}_{2^j}*h_{j}
	\Big\|_{\ell^{r}}
	\le C_{r} (\log J) |F|^{1/r}
	$$
	where $h_{j}= g_j \ind{F} \ge 0$.

	We partition the set $(J, 2J] \cap \ZZ$ into at most $2 \mu (\log_2 J)$ subsets $S$ with
	the sparseness property
	\begin{align}
		\label{eq:44}
		l, l' \in S, \quad \text{if} \quad l \neq l' \quad \text{then} \quad
		\abs{l - l'} \geq 2\mu (\log_2 J)
	\end{align}
	where $\mu > 0$ is a constant satisfying \eqref{eq:37}. Therefore, it is enough
	to prove that for each integer $r \geq 2$
	\begin{align}
		 \label{eq:45}
		 \Big\|\sum_{j\in S}\widetilde{K}_{2^j}*h_{j}\Big\|_{\ell^{r}}
		 \le C_{r}
		 |F|^{1/r}.
	\end{align}
	We show \eqref{eq:45} by induction with respect to $r$. For $r=2$ we have
	$$
	\Big\| \sum_{j\in S} \widetilde{K}_{2^j}*h_{j} \Big\|_{\ell^{2}}
	\le
	|F|^{1/2}
	\sup_{\|f\|_{\ell^{2}}=1}
	\Big\| \sum_{j\in S}(K_{2^j}*f)g_{j} \Big\|_{\ell^{2}}
	\le
	|F|^{1/2}
	\sup_{\|f\|_{\ell^{2}}=1}
	\big\| \sup_{j\in S}K_{2^j}*f \big\|_{\ell^{2}}\\
	\lesssim
	|F|^{1/2}
	$$
	where in the last step we have used Theorem \ref{thm:7}. For $r > 2$ we expand
	the left-hand side of \eqref{eq:45}. There is a constant $C_r > 0$, which may depend only on
	$r$ and such that
	\begin{equation}
		\label{eq:46}
		\Big\|\sum_{j\in S}\widetilde{K}_{2^j}*h_{j}\Big\|_{\ell^{r}}^{r}
		\leq
		C_r
		\sum_{J<j_1<\ldots<j_{r}\le 2J}
		\sum_{x\in\ZZ^d}
		\prod_{n=1}^{r}
		\widetilde{K}_{2^{j_n}}*h_{j_n}(x)
		+
		C_r \sum_{x\in\ZZ^d} \Big(\sum_{j\in S}\widetilde{K}_{2^{j}}*h_{j}(x)\Big)^{r-1}.
	\end{equation}
	To treat the first term in \eqref{eq:46} we need to prove that for any increasing sequence
	$J = j_0 < j_1 < \ldots < j_r \leq 2J$
	\begin{align}
		\label{eq:47}
		\Big\|\Big(\prod_{n=2}^{r} \widetilde{K}_{2^{j_n}}*h_{j_n}\Big)
		*\big(K_{2^{j_1}}-K_{2^{j_0}}\big)\Big\|_{\ell^2}
		\lesssim
		J^{-r} |F|^{1/2}
	\end{align}
	Assuming momentarily \eqref{eq:47}  we would have
	$$
	\Big\lvert
	\sum_{x \in \ZZ^d}
	\big(\widetilde{K}_{2^{j_1}} - \widetilde{K}_{2^{j_0}}\big)*h_{j_1}(x)
	\prod_{n=2}^{r}
	\widetilde{K}_{2^{j_n}}*h_{j_n}(x)
	\Big\rvert
	\le
	|F|^{1/2}
	\Big\|
	\Big(\prod_{n=2}^{r}\widetilde{K}_{2^{j_n}}*h_{j_n}\Big)
	*\big(K_{2^{j_1}}-K_{2^{j_0}}\big)
	\Big\|_{\ell^2}
	\lesssim J^{-r}|F|,
	$$
	thus
	\begin{multline}
		\label{eq:49}
		\sum_{J<j_1<\ldots<j_{r}\le 2J}
		\sum_{x\in\ZZ^d}
		\prod_{n=1}^{r}
		\widetilde{K}_{2^{j_n}}*h_{j_n}(x)\\
		\leq
		|F|
		+\sum_{J<j_1<\ldots<j_{r}\le 2J}
		\sum_{x\in\ZZ^d}
		\widetilde{K}_{2^{j_0}}*h_{j_1}(x)
		\prod_{n=2}^{r} \widetilde{K}_{2^{j_n}}*h_{j_n}(x)\\
		\leq
		|F|
		+\sum_{x\in\ZZ^d}
		\Big(\widetilde{K}_{2^{j_0}}*\sum_{j\in S}h_{j}\Big)(x)
		\Big(\sum_{j\in S} \widetilde{K}_{2^{j}}*h_{j}(x)\Big)^{r-1} \\
		\leq
		|F|
		+\sum_{x\in\ZZ^d}
		\Big(\sum_{j\in S}\widetilde{K}_{2^{j}}*h_{j}(x)\Big)^{r-1}.
	\end{multline}
	Hence, by the inductive hypothesis we obtain
	$$
	\Big\|
	\sum_{j\in S}
	\widetilde{K}_{2^j}*h_{j}
	\Big\|_{\ell^{r}}^{r}
	\lesssim
	|F|
	+\Big\|\sum_{j\in S}\widetilde{K}_{2^{j}}*h_{j}\Big\|_{\ell^{r-1}}^{r-1}
	\lesssim |F|.
	$$
	This completes the proof and shows that \eqref{eq:45} holds.
	
	It remains to prove the bound \eqref{eq:47}. First, we introduce approximating multipliers
	$$
	\Upsilon_N^t(\xi)
	=
	\sum_{s = 0}^t
	\sum_{a/q\in \mathscr{R}_s}
	G(a/q)
	\Phi_N(\xi-a/q)
	\eta_s(\xi-a/q)
	\vrho_N(\xi-a/q)
	$$
	where $\vrho_N(\xi)=\eta\big((J^{-\mu} N)^A\xi\big)$. Then, by Lemma \ref{lem:4} and estimates
	\eqref{eq:8} and \eqref{eq:9}
	\begin{align}
		\label{eq:58}
		\big| m_{N}(\xi)-\Upsilon_N^t(\xi)\big|
   		\lesssim
		N^{-\delta_1}+2^{-t \delta}+ J^{-\mu/d}.
	\end{align}
	Moreover,
	\begin{align}
		\label{eq:59}
		\big\|\calF^{-1}\big(\Upsilon_N^t \big)\big\|_{\ell^1} \lesssim 2^{t(d+1)}.
	\end{align}
	For each $n \in \NN$ we set
	$$
	t_n= \max\big\{r, (2 d + 3) \delta^{-1} \big\}^n (\log_2 J)
	$$
	By Plancherel's theorem and \eqref{eq:58} we have
	\begin{align}
		\label{eq:51}
		\big\|\widetilde{K}_{2^{j_n}}*h_{j_n}
		-\calF^{-1}\big(\overline{\Upsilon}_{2^{j_n}}^{t_n} \hat{h}_{j_n} \big)
		\big\|_{\ell^2}
		\lesssim
		\big(2^{-j_n\delta_1}+2^{-t_n \delta}+ J^{-\mu/d}\big)
		|F|^{1/2}.
	\end{align}
	Moreover, by \eqref{eq:59}, for every $x\in \ZZ^d$ we have
	\begin{align}
		\label{eq:52}
		\big|
		\calF^{-1}\big(\overline{\Upsilon}_{2^{j_n}}^{t_n} \hat{h}_{j_n} \big)(x)
		\big|
		\le
		\big\|
		\calF^{-1}\big(\overline{\Upsilon}_{2^{j_n}}^{t_n} \big)
		\big\|_{\ell^1}
		\leq
		C2^{t_n(d+1)}.
	\end{align}
	Let us denote by $\calW$ the support of
	$\big(\overline{\Upsilon}_{2^{j_2}}^{t_2} \hat{h}_{j_2}\big) * \ldots *
	\big(\overline{\Upsilon}_{2^{j_r}}^{t_r} \hat{h}_{j_r}\big)$. Then
	$$
	\calW \subseteq
	\bigcup_{q = 1}^{2^{r t_r}}
	\bigcup_{a \in \NN_q^d}
	\big\{\xi \in \TT^d : \lvert \xi_\gamma - a_\gamma/q\rvert <
		2^{- j_2\abs{\gamma}} J^{\mu \abs{\gamma}} \text{ for all } \gamma \in \Gamma\big\}.
	$$
	Furthermore, by H\"{o}lder's inequality and \eqref{eq:52} we have
	\begin{multline}
		\label{eq:54}
		\Big\|
		\prod_{n=2}^r
		\calF^{-1} \big(\overline{\Upsilon}_{2^{j_n}}^{t_n} \hat{h}_{j_n}\big)
		\Big\|_{\ell^{2}}
		\leq
		\prod_{n=2}^{r}
		\big\|\calF^{-1}
		\big(\overline{\Upsilon}_{2^{j_n}}^{t_n} \hat{h}_{j_n}\big)
		\big\|_{\ell^{2(r-1)}}
		\leq
		\prod_{n=2}^{r}
		\big\|\calF^{-1}\big(\overline{\Upsilon}_{2^{j_n}}^{t_n} \big)\big\|_{\ell^{1}}
	    \prod_{n=2}^{r}
		\|h_{j_n}\|_{\ell^{2(r-1)}}\\
		\lesssim
		2^{(t_2+\ldots+t_{r})(d+1)}|F|^{1/2}
		\lesssim
		2^{2 t_r (d+1)} |F|^{1/2}
	\end{multline}
	because $t_2 + \ldots + t_r \leq 2 t_r$. Next, by \eqref{eq:44} we have
	$j_2 - j_1 \geq 2\mu (\log_2 J)$ thus
	$$
	J^\mu 2^{-j_2} \leq 2^{-j_1}.
	$$
	In particular, if $\xi \in \calW$ then there are $1 \leq q \leq 2^{r t_r}$ and $a \in \NN_q^d$
	such that for each $\gamma\in \Gamma$ we have
	\begin{align*}
		\Big|\xi_{\gamma}-\frac {a_{\gamma}}{q}\Big|
		\leq J^{\mu \abs{\gamma}} 2^{- j_2\abs{\gamma} }
		\leq 2^{-j_1(\abs{\gamma} - \beta)}.
	\end{align*}
	Since $2^{r t_r} \leq 2^{\alpha J}$ for sufficiently large $J$, by Proposition \ref{prop:1} and
	\eqref{eq:10}, we have
	\begin{multline}
		\label{eq:55}
		\big|m_{2^{j_1}}(\xi) - m_{2^{j_0}}(\xi)\big|
		\leq
		|\Phi_{2^{j_1}}(\xi-a/q)-1| + |\Phi_{2^{j_0}}(\xi-a/q)-1| + 2^{-j_0/4} + 2^{-j_1/4}\\
		\lesssim
		\big\lvert 2^{j_1 A}(\xi - a/q) \big\rvert
		+\big\lvert 2^{j_0 A}(\xi - a/q) \big\rvert
		+2^{-J/4}
		\lesssim
		J^{- \mu}.
	\end{multline}
	Next, we may estimate
	\begin{multline}
		\label{eq:60}
		\Big\|\Big(\prod_{n=2}^{r}\widetilde{K}_{2^{j_n}}*h_{j_n}\Big)
		*\big(K_{2^{j_1}}-K_{2^{j_0}}\big)\Big\|_{\ell^2}\\
		\leq
		\sum_{n=2}^{r}
		\prod_{k=2}^{n-1}
		\big\|
		\calF^{-1}\big(\overline{\Upsilon}_{2^{j_{k}}}^{t_{k}} \hat{h}_{j_{k}}\big)
		\big\|_{\ell^{\8}}
		\cdot
		\big\|
		\widetilde{K}_{2^{j_n}}*h_{j_n}
		- \calF^{-1} \big(\overline{\Upsilon}_{2^{j_n}}^{t_n} \hat{h}_{j_n}\big)
		\big\|_{\ell^{2}}\\
		+ \Big\|\prod_{n=2}^{r}
		\calF^{-1}\big(\overline{\Upsilon}_{2^{j_n}}^{t_n} \hat{h}_{j_n}\big)
		\Big\|_{\ell^{2}}
		\cdot
		\sup_{\xi \in \calW}
		\big|m_{2^{j_1}}(\xi)-m_{2^{j_0}}(\xi)\big|.
	\end{multline}
	By \eqref{eq:51} and \eqref{eq:52} the first term in \eqref{eq:60} is bounded by
	\begin{multline*}
		\sum_{n=2}^{r}
		\Big(\prod_{k = 2}^{n-1} 2^{(d+1) t_{k}}\Big)
		\big(2^{-\delta t_n} + 2^{-j_n\delta_1} + J^{-\mu/d} \big)
		|F|^{1/2}\\
		\lesssim
		\sum_{n=2}^{r}
		2^{2(d+1) t_{n-1} - \delta t_n}
		|F|^{1/2}
		+2^{2(d+1) t_r} \big( 2^{-J\delta_1} + J^{-\mu/d}\big) |F|^{1/2}\\
		\lesssim
		2^{-t_1}
		|F|^{1/2}
		+2^{2(d+1) t_r} \big( 2^{-J\delta_1} + J^{-\mu/d}\big) |F|^{1/2}.
	\end{multline*}
	Moreover, by \eqref{eq:54} and \eqref{eq:55} the second term in \eqref{eq:60} is
	bounded by $2^{2(d+1) t_r} J^{- \mu} |F|^{1/2}$. Since $2^{-t_1} \leq J^{-r}$ it is enough
	to select $\mu$ satisfying
	\begin{equation}
		\label{eq:37}
		\mu \geq 2 (d+1)^3 \max\big\{r , (2d+3) \delta^{-1}\big\}^r.
	\end{equation}
	This completes the proof of Theorem \ref{th:5}.
\end{proof}

\section{Maximal theorem}\label{sec4}
We are ready to proof Theorem \ref{thm:1}. In view of Lemma \ref{lem:1} it is enough to show
\begin{theorem}
	\label{th:3}
	Let $p \in (1, \infty)$. There is $C_p > 0$ such that for every
	$f \in \ell^p\big(\ZZ^d\big)$
	\begin{equation}
		\label{eq:27}
		\big\lVert
		\sup_{N \in \NN} \big\lvert M_N^\calQ f \big\rvert
		\big\rVert_{\ell^p}
		\leq
		C_p
		\vnorm{f}_{\ell^p}.
	\end{equation}
\end{theorem}
	Let us observe that the supremum in \eqref{eq:27} may be restricted to the set of dyadic
	numbers $\mathcal{D}$. As we mentioned in the introduction we shall explore restricted
	interpolation lemma of Ionescu and Wainger introduced in \cite{IW} (see also \cite{IMSW}).
	Namely,
	\begin{lemma}
		\label{lem:0}
		Suppose for each $r \in (1,2]$, $\epsilon \in (0, 1]$ and $\lambda > 0$ there is a
		sequence of linear operators $\big(A_j^{\lambda,\epsilon}: j \in \NN\big)$ such that
		$$
		\big\lVert \sup_{j \in \NN}
		\big\lvert
		A_j^{\lambda,\epsilon} f
		\big\rvert
		\big\rVert_{\ell^r}
		\leq
		C_{\epsilon, r} \lambda^\epsilon \vnorm{f}_{\ell^r}
		\quad \text{and} \quad
		\big\lVert \sup_{j \in \NN}
		\big\lvert
		M_{2^j} f - A_j^{\lambda, \epsilon} f
		\big\rvert
		\big\rVert_{\ell^2}
		\leq
		D_\epsilon \lambda^{-1} \vnorm{f}_{\ell^2}.
		$$
		Then for each $p \in (1, 2]$ there exists a constant $C_p > 0$
		$$
		\big\lVert \sup_{j \in \NN} \big\lvert M_{2^j} f \big\rvert \big\rVert_{\ell^p}
		\leq
		C_p \vnorm{f}_{\ell^p}.
		$$
	\end{lemma}
\begin{proof}[Proof of Theorem \ref{th:3}]
	Let $\epsilon \in (0, 1]$.  If $\lambda \leq 1$ then we may take $A_j^{\lambda, \epsilon} = 0$
	since by Theorem \ref{thm:7}
	\begin{align*}
		\big\lVert \sup_{j \in \NN}
		\big\lvert
		M_{2^j} f - A_j^{\lambda, \epsilon} f
		\big\rvert
		\big\rVert_{\ell^2}
		\leq
		C \lambda^{-1} \vnorm{f}_{\ell^2}.
	\end{align*}
	For $\lambda > 1$ we choose $t \in \NN$ such that
	$$
	t = \big\lfloor \delta_4^{-1} \log_2 \lambda \big\rfloor + 1
	$$
	where $\delta_4>0$ is the exponent from Theorem \ref{th:4}. Let $\kappa_t$ be defined
	by \eqref{eq:77}. If $j < 2^{\kappa_t}$ then we set $A_j^{\lambda, \epsilon} = M_{2^j}$.
	By Theorem \ref{th:5} we may write
	$$
	\big\lVert
	\sup_{j < 2^{\kappa_t}}
	\big\lvert
	M_{2^j} f
	\big\rvert
	\big\rVert_{\ell^r}
	\leq
	\sum_{k = 1}^{\kappa_t}
	\big\lVert
	\sup_{2^{k-1} \leq j < 2^k}
	\big\lvert
	M_{2^j} f
	\big\rvert
	\big\rVert_{\ell^r}
	\lesssim
	\sum_{k = 1}^{\kappa_t}
	k \cdot \vnorm{f}_{\ell^r}
	\lesssim
	t^2 \cdot \vnorm{f}_{\ell^r}
	\lesssim \lambda^{\epsilon}\vnorm{f}_{\ell^r}.
	$$
	For $j \geq 2^{\kappa_t}$ we define $A_j^{\lambda, \epsilon}=\Omega_{2^j}^t$. Then by
	Theorem \ref{th:1} and Theorem \ref{th:4} we have
	\begin{align*}
		\big\lVert
		\sup_{j \geq 2^{\kappa_t}}
		\big\lvert
		\mathcal{F}^{-1}\big(\Omega_{2^j}^t \hat{f}\big)
		\big\rvert
		\big\rVert_{\ell^r}
		\lesssim \vnorm{f}_{\ell^r}
		\quad \text{and} \quad
		\big\lVert
		\sup_{j \geq 2^{\kappa_t}}
		\big\lvert
		M_{2^j} f - \mathcal{F}^{-1}\big(\Omega_{2^j}^t \hat{f}\big)
		\big\rvert
		\big\rVert_{\ell^2}
		\lesssim \lambda^{-1} \vnorm{f}_{\ell^2}.
	\end{align*}
	This completes the proof of Theorem \ref{th:3}.
\end{proof}

\section{Variational theorem}
\label{sec5}
In this section we prove Theorem \ref{thm:2}. Again, using Lemma \ref{lem:1} it is enough
show
\begin{theorem}
	\label{th:7}
	Let $p \in (1, \infty)$ and $r > \max\{p, p/(p-1)\}$. There is $C_{p,r} > 0$ such that for each
	$f \in \ell^p\big(\ZZ^d\big)$
	\begin{equation}
		\label{eq:79}
		\big\lVert
		\var{\big(\seq{M_N^\calQ f}{N \in \NN}\big)}
		\big\rVert_{\ell^p}
		\le C_{p, r}
		\vnorm{f}_{\ell^p}.
	\end{equation}
\end{theorem}

\begin{proof}
	We only prove \eqref{eq:79} for $p=2$. In order to obtain \eqref{eq:79}  for $p \in (1, \infty)$ and $r > \max\{p, p/(p-1)\}$ it suffices to repeat the argument form \cite{K}, and interpolate the estimate \eqref{eq:27} with the  estimate \eqref{eq:79} for $p = 2$. To prove the inequality \eqref{eq:79} for $p = 2$,
	we will make use of the estimate \eqref{eq:32} and separately treat long and short
	variations.
	
	\noindent
	{\textbf{Long variations:}}
	In order to estimate long variations we shall use Theorem \ref{thm:6}. Indeed, for $r>2$
	\begin{multline*}
		\big\|V_r^L\big(M_Nf: N\in\NN\big)\big\|_{\ell^2}
		=\big\|V_r\big(M_{2^j}f: j\ge0\big)\big\|_{\ell^2}\\
		\le
		\sum_{s\ge0}
		\big\|V_r\big(\mathcal{F}^{-1}\big(\nu_{2^j}^s\hat{f}\big): j\ge 0\big) \big\|_{\ell^2}
		+\Big\|V_r\Big(M_{2^j}f
		-\sum_{s\ge0}
		\mathcal{F}^{-1}\big(\nu_{2^j}^s\hat{f}\big): j\ge 0\Big)\Big\|_{\ell^2}\\
		\lesssim
		\sum_{s\ge0}
		\big\|V_r\big(\mathcal{F}^{-1}\big(\nu_{2^j}^s\hat{f}\big): j\ge 0\big)\big\|_{\ell^2}
		+\bigg(
		\sum_{j\ge0}\big\|M_{2^j}f
		-\sum_{s\ge0}\mathcal{F}^{-1}\big(\nu_{2^j}^s\hat{f}\big)\big\|^2_{\ell^2}
		\bigg)^{1/2}\\
		\lesssim \sum_{s\ge0}2^{-s\delta_2}\|f\|_{\ell^2}+\sum_{j\ge0}2^{-j\delta_1}\|f\|_{\ell^2}
	\end{multline*}
	where the penultimate inequality follows from \eqref{eq:18} and the last inequality
	is guaranteed by Theorem \ref{thm:6} and Lemma \ref{lem:4}.

	\noindent
	{\textbf{Short variations:}}
	Let us define the Fourier projection $\Pi_Q$ onto a set $Q\subseteq\TT^d$ by setting
	$$
	\Pi_Q f= \calF^{-1}\big(\ind{Q} \hat{f}\big)
	$$
	and observe that according to the definition of short variations we have
	\begin{multline}
		\label{eq:36}
		\big\|V_r^S\big(M_N f: N\in\NN\big)\big\|_{\ell^2}^2
		\lesssim
		\sum_{n \ge 0}
		\big\|
		V_2\big(M_N\big(\Pi_{\mathfrak{M}_{2^n}} f\big) : 2^n \leq N < 2^{n+1}\big)
		\big\|_{\ell^2}^2\\
		+\sum_{n \ge 0}
		\big\|
		V_2\big(M_N\big(\Pi_{\mathfrak{m}_{2^n}} f\big): 2^n \leq N < 2^{n+1}\big)
		\big\|_{\ell^2}^2
	\end{multline}
	for major $\mathfrak{M}_{2^j}$ and minor $\mathfrak{m}_{2^j}$ arcs defined in Section
	\ref{sec3}. The proof of Theorem \ref{th:7} will be completed if we show that the sums in
	\eqref{eq:36} can be dominated by $\|f\|^2_{\ell^2}$. Applying \eqref{eq:39} we get the desired
	bound for the second sum in \eqref{eq:36}. Indeed, by Plancherel's theorem and  Weyl's inequality 
    \cite{SW0},  we have
	$$
	\big\lVert M_N\big(\Pi_{\mathfrak{m}_{2^n}} f\big) \big\rVert_{\ell^2}
	\leq
	\sup_{\xi \in \mathfrak{m}_{2^n}} \lvert m_N(\xi) \rvert \cdot
	\big\lVert \Pi_{\mathfrak{m}_{2^n}} f \big\rVert_{\ell^2}
	\lesssim
	2^{-n\delta' }
	\big\lVert \Pi_{\mathfrak{m}_{2^n}} f \big\rVert_{\ell^2}
	$$
	and
	$$
	\big\lVert
	M_{N+1}\big(\Pi_{\mathfrak{m}_{2^n}} f\big) - M_{N}\big(\Pi_{\mathfrak{m}_{2^n}} f\big)
	\big\rVert_{\ell^2}
	\lesssim 2^{-n}
	\big\lVert \Pi_{\mathfrak{m}_{2^n}} f \big\rVert_{\ell^2}.
	$$
	Therefore, using inequality \eqref{eq:39} we obtain
	$$
	\sum_{n \ge 0}
	\big\|V_2\big(M_N(\Pi_{\mathfrak{m}_{2^n}}f): 2^n \leq N < 2^{n+1}\big)\big\|_{\ell^2}^2
	\lesssim
	\sum_{n \ge 0}
	2^{- n \delta'}
	\big\lVert \Pi_{\mathfrak{m}_{2^n}}f \big\rVert_{\ell^2}^2
	\lesssim
	\lVert f \rVert_{\ell^2}^2.
	$$
	To deal with the first sum in \eqref{eq:36} we need to have a more subtle decomposition of
	the family of major arcs. For $u \geq -\beta n$ let us define
	$$
	\mathfrak{N}_{2^n}^u(a/q) = \mathfrak{M}_{2^n}^u(a/q) \setminus \mathfrak{M}_{2^n}^{u+1}(a/q)
	$$
	where
	$
	\mathfrak{M}_{2^n}^u(a/q)
	=
	\big\{\xi \in \TT^d :
		\lvert
		\xi_\gamma - a_\gamma / q
		\rvert
		\leq
		2^{-n\abs{\gamma} - u}
		\text{ for all } \gamma \in \Gamma
	\big\}.
	$
	Setting
	$$
	\mathfrak{N}_{2^n, s}^u
	=
	\bigcup_{a/q \in  \mathscr{R}_s}
	\mathfrak{N}_{2^n}^u(a/q).
	$$
	we may write
	\begin{align}
		\label{eq:35}
		\mathfrak{M}_{2^n}
		=\bigcup_{0 \leq s \leq \alpha n - 1}
		\bigcup_{a/q \in \mathscr{R}_s}
		\mathfrak{M}_{2^n}(a/q)
		=\bigcup_{0 \leq s \leq \alpha n - 1}
		\bigcup_{u \ge - \beta n}
		\mathfrak{N}_{2^n, s}^u.
	\end{align}
	Let $2^n \leq N < 2^{n+1}$. If $\xi \in \mathfrak{M}_{2^n} \cap \mathfrak{N}_{2^n}^u(a/q)$ for
	$a/q \in \mathscr{R}_s$ then for all $\gamma \in \Gamma$
	$$
	\Big\lvert \xi_\gamma - \frac{a_\gamma}{q} \Big\rvert
	\leq 2^d N^{-\abs{\gamma} + \beta}.
	$$
	Thus, by Proposition \ref{prop:1} and \eqref{eq:8} together with \eqref{eq:9} and
	\eqref{eq:10}, we get
	\begin{equation}
		\label{eq:40}
		\big\lvert m_N(\xi) - m_{2^n}(\xi) \big\rvert
		\lesssim
		2^{-n/4} + 2^{-\delta s} \big\lvert \Phi_N(\xi-a/q) - \Phi_{2^n}(\xi - a/q)\big\rvert
		\lesssim
		2^{-\delta s} \big(2^{-\abs{u}/d} + 2^{-n/8}\big)
	\end{equation}
	provided that $8 \alpha \delta \leq 1$. Next, we set
	\begin{align*}
		\widetilde{\mathfrak{N}}_{2^n, s}
		=
		\bigcup_{u > d n /8}
		\mathfrak{N}_{2^n, s}^u
	\end{align*}
	and observe that, by \eqref{eq:35}, we get
	\begin{multline*}
		\big\lVert
		V_r\big(
		M_N\big(\Pi_{\mathfrak{M}_{2^n}}f\big) : 2^n \leq N < 2^{n+1}
		\big)
		\big\rVert_{\ell^2}
		=
		\big\lVert
		V_r\big(\calF^{-1} \big((m_N - m_{2^n}) \ind{\mathfrak{M}_{2^n}} \hat{f}\big)
		: 2^n \leq N < 2^{n+1}\big) \big\rVert_{\ell^2}\\
		\lesssim
		\sum_{0 \leq s \leq \alpha n - 1}
		\sum_{-\beta n \leq u \leq d n /8}
		\big\lVert
		V_2 \big(\calF^{-1}\big((m_N - m_{2^n}) \ind{\mathfrak{N}_{2^n, s}^u} \hat{f}\big):
		2^n \leq N < 2^{n+1}\big)
		\big\rVert_{\ell^2}\\
		+
		\sum_{0 \leq s \leq \alpha n - 1}
		\big\lVert
		V_2\big(\calF^{-1}\big(
		(m_N - m_{2^n}) \ind{\widetilde{\mathfrak{N}}_{2^n, s}^u} \hat{f}
		\big): 2^n \leq N < 2^{n+1}\big)
		\big\rVert_{\ell^2}.
	\end{multline*}
	If $-\beta n \leq u \leq d n/8$ then using \eqref{eq:39} and \eqref{eq:40} we can estimate
	$$
	\big\lVert
	V_2 \big(\calF^{-1}\big((m_N - m_{2^n}) \ind{\mathfrak{N}_{2^n, s}^u} \hat{f}\big):
	2^n \leq N < 2^{n+1}\big)
	\big\rVert_{\ell^2}
	\lesssim
	2^{-\delta s/2}
	2^{-|u|/(2d)}
	\big\lVert
	\Pi_{\mathfrak{N}_{2^n, s}^u} f
	\big\rVert_{\ell^2},
	$$
	otherwise
	$$
	\big\lVert
	V_2\big(\calF^{-1}\big(
	(m_N - m_{2^n}) \ind{\widetilde{\mathfrak{N}}_{2^n, s}^u} \hat{f}
	\big): 2^n \leq N < 2^{n+1}\big)
	\big\rVert_{\ell^2}
	\lesssim
	2^{-\delta s/2}
	2^{-n/16}
	\big\lVert
	\Pi_{\widetilde{\mathfrak{N}}_{2^n, s}^u} f
	\big\rVert_{\ell^2}.
	$$
	Therefore, by Cauchy--Schwarz inequality we get
	\begin{multline*}
		\big\lVert
		V_r\big(
		M_N\big(\Pi_{\mathfrak{M}_{2^n}}f\big) : 2^n \leq N < 2^{n+1}
		\big)
		\big\rVert_{\ell^2}^2
		\lesssim
		\sum_{0 \leq s \leq \alpha n - 1}
		2^{-\delta s/2}
		\sum_{u \geq -\beta n}
		2^{-\abs{u}/(2d)}
		\big\lVert
		\Pi_{\mathfrak{N}_{2^n, s}^u} f
		\big\rVert_{\ell^2}^2\\
		+
		2^{-n/8}
		\sum_{s \geq 0}
		2^{-\delta s/2}
		\big\lVert
		\Pi_{\widetilde{\mathfrak{N}}_{2^n, s}} f
		\big\rVert_{\ell^2}^2.
	\end{multline*}
	Next, we have
	\begin{multline*}
		\sum_{n \geq 0}
		\sum_{0 \leq s \leq \alpha n - 1}
		2^{-\delta s/2}
		\sum_{u \geq -\beta n}
		2^{-\abs{u}/(2d)}
		\big\lVert
		\Pi_{\mathfrak{N}_{2^n, s}^u} f
		\big\rVert_{\ell^2}^2
		\lesssim
		\sum_{s \geq 0}
		2^{-\delta s/2}
		\sum_{n \geq s / \alpha}
		\sum_{u \geq -\beta n}
		2^{-\abs{u}/(2d)}
		\big\lVert
		\Pi_{\mathfrak{N}_{2^n, s}^u} f
		\big\rVert_{\ell^2}^2\\
		\lesssim
		\sum_{s \geq 0}
		2^{-\delta s/2}
		\sum_{u \in \ZZ}
		2^{-\abs{u}/(2d)}
		\sum_{n \geq \max\{s/\alpha, -u/\beta\}}
		\big\lVert
		\Pi_{\mathfrak{N}_{2^n, s}^u} f
		\big\rVert_{\ell^2}^2
		\lesssim
		\lVert f \rVert_{\ell^2}^2
	\end{multline*}
	where the last inequality follows since $\mathfrak{N}_{2^n, s}^u$ are disjoint for $n \geq \max\{s/\alpha, -u/\beta\}$ while $u\in\ZZ$ and $s\in\NN_0$ are fixed. Hence,
	$$
	\sum_{n \geq 0}
	\big\lVert
	V_r\big(
	M_N\big(\Pi_{\mathfrak{M}_{2^n}}f\big) : 2^n \leq N < 2^{n+1}
	\big)
	\big\rVert_{\ell^2}^2
	\lesssim
	\lVert f \rVert_{\ell^2}^2
	+
	\sum_{n \geq 0}
	2^{-n/8}
	\lVert f \rVert_{\ell^2}^2
	\lesssim
	\lVert f \rVert_{\ell^2}^2.
	$$
	This provides the bound for the second sum in \eqref{eq:36} and completes the proof of Theorem
	\ref{th:7} for $p=2$ and $r > 2$.
\end{proof}

\begin{bibliography}{discrete}
	\bibliographystyle{amsplain}
\end{bibliography}

\end{document}